\documentclass[preprint,12pt]{elsarticle}
\usepackage{amsmath,amssymb}
\usepackage{algorithm}
\usepackage{algorithmic}
\usepackage{lineno,hyperref}
\usepackage{graphicx}
\usepackage{multirow}
\usepackage{soul}
\usepackage{epstopdf}
\usepackage{dsfont}
\usepackage{bm}
\usepackage{color}

\biboptions{numbers,sort&compress}

\newtheorem{thm}{Theorem}[section]
\newtheorem{prop}{Proposition}[section]

\makeatletter
\newtheorem{theorem}{Theorem}[section]
\newtheorem{lemma}[theorem]{Lemma}

\newtheorem{remark}{Remark}[section]

\newtheorem{proposition}[theorem]{Proposition}
\newenvironment{proof}[1][Proof]{\noindent\textbf{#1.} }{\hfill $\Box$}
\numberwithin{equation}{section}

\allowdisplaybreaks[4]
\begin{document}

\begin{frontmatter}

\title{Existence and Spatial Decay of Forced Waves for the Fisher-KPP Equation
with a Degenerate Shifting Environment}
\author[DANWEI1,DANWEI2]{Zhibao Tang}\ead{tangzhibao1022@163.com}
\author[DANWEI1]{Shi-Liang Wu}\ead{slwu@xidian.edu.cn}
\author[DANWEI2]{Yaping Wu\corref{cor}}\ead{yaping\_wu@hotmail.com}
\cortext[cor]{Corresponding author.}
\address[DANWEI1]{School of Mathematics and Statistics, Xidian University, Xi'an, Shaanxi 710071, P.R. China}
\address[DANWEI2]{School of Mathematical Sciences, Capital Normal University,
Beijing 100048, China}

\begin{abstract}
This paper is concerned with the existence, the multiplicity and the precise spatial decay of forced waves to the following
heterogeneous Fisher-KPP equation
with a degenerate shifting environment
$$
u_t=u_{xx}+u\big{(}a(x-ct)-u\big{)},~~t>0,~x\in \mathbb{R},
$$
where $c>0$ is a shifting speed of the resource and $a(z)$ can be any given positive function satisfying $a(-\infty)=\alpha>0=a(+\infty),~0<a(z)\leq\alpha,~z\in\mathbb{R}$ and $a'(z)\le 0$ for $z\gg 1$. By applying more general ODE asymptotic theories and  the detailed asymptotic analysis to all the local  positive solutions $\psi(z)$ near $z=+\infty$, which satisfy the nonlinear differential  equation  $\psi''(z)+c\psi'(z)+\psi(z)(a(z)-\psi(z))=0$ for $z\ge z_0$ and $\psi(+\infty)=0$, it can be proved  that there always exist a family of local positive solutions decaying exponentially at $z=+\infty$. We also prove that if
$a(z)$ decays to zero fast enough near $z\to+\infty$  such that $\int^{+\infty}_{z_0}\widetilde{\sigma}_2(z){\rm d}z =+\infty$ with $\widetilde{\sigma}_2(z)=\exp\big(\int^z_{z_0}\frac{-c+\sqrt{c^2-4a(s)}}{2}{\rm d}s\big)$, then there exist no local positive solutions decaying non-exponentially at $z=+\infty$;
 however if $a(z)$ decays slowly enough such that  $\int^{+\infty}_{z_0}\widetilde{\sigma}_2(z){\rm d}z<+\infty$, then there exist two  types of local positive solutions decaying non-exponentially at $z=+\infty$. The precise decay rates of all the local positive solutions are also obtained.

By virtue of the precise decaying estimates  and the ordering of all forced  waves, further by  constructing appropriate pairs of sub-/super solutions and by applying variational method or sliding techniques, we prove that
the above heterogeneous  Fisher-KPP equation admits precisely one or infinitely many or several families of forced waves or no forced waves  depending on the  range of $c$ and
the decaying rate  of $a(z)$ at $z=+\infty$.  Our results also indicate that for any given degenerate $a(z)$ and any $c>0$ the existence/nonexistence of forced waves decaying non-exponentially at $z=+\infty$ are solely determined by  the convergence/divergence of  the integral $\int^{+\infty}_{z_0}\widetilde{\sigma}_2(z){\rm d}z$. Precisely speaking,  we prove that there exists no forced wave decaying exponentially when $c\ge 2\sqrt{\alpha}$; and  there exists a  forced wave solution decaying exponentially  for each $c\in (0,2\sqrt{\alpha})$, which is also the unique  or the minimal forced wave.   For the case  $ \int^{+\infty}_{z_0}\widetilde{\sigma}_2(z){\rm d}z<+\infty$, we prove that for any $c>0$ there exist infinitely many forced waves decaying non-exponentially and only the  maximal forced wave is not in $L^1([z_0,+\infty))$. We also obtained the precise decaying rates of all the forced waves at $z=+\infty$.
\noindent
\noindent

\medskip

\end{abstract}


\end{frontmatter}


\section{Introduction}
\label{}
For spatially diffusive population models in
the whole space, traveling  wave solutions are the simplest entire solutions which can describe some typical  wave phenomena  or the asymptotic propagating behavior of spatial segregation of population.
One species population model with a shifting environment can be described by
\begin{equation}\label{origequa1}
u_t=u_{xx}+f(x-ct,u),\quad t>0,~x\in \mathbb{R},
\end{equation}
where $c\in\mathbb{R}$ is the shifting speed of the environment.
This type of equations usually model
the population segregation in a shifting media.
The dynamics of solutions to (\ref{origequa1}), including
the longtime behavior and asymptotic speeds of solutions, the existence and stability of forced waves,
were deeply studied in recent years,
 see \cite{[BN],[BR2],[F1],[F],[LY],[HT],[VHH],[YCW]} and the
references therein.
The effect of climate change and the worsening
of environment resulting from industrialization on population dynamics,
such as extinction, persistence, population migration and some new propagating phenomena
are interesting in both math and application \cite{[AB],[BDNZ],[BL],[HL],[CC],[PL],[LBSW],[TCD]}.
In application, if $u(t,x)$ represents the infected population density, the existence/non-existence of traveling waves to (\ref{origequa1})
can partly answer the epidemiological question about whether the infection can keep pace with the host expansion
\cite{[CTW],[FLW]}.

The model (\ref{origequa1}) with a shifting environment was proposed by Berestycki et al. \cite{[BDNZ]} to
study the impact of climate shift on the dynamics of a biological species.
In \cite{[BDNZ]},
a typical case of $f(z,u)$ can be in the form $f(z,u)=ug(z,u)$ with $g(z,u)<0$ for $|z|\geq L\gg1$, such as
\begin{equation*}
g(z,u)=
\begin{cases}
-\widetilde{r},~~~~~~~~~~~x<0~{\rm and}~x>L,\\
r(1-\frac{u}{K}),~~~~0\leq x\leq L,
\end{cases}
\end{equation*}
for given positive parameters $\widetilde{r},~r,~K$ and $L$,
which means that a favorable
habitat is surrounded by unfavorable ones.
The authors proved that the existence/non-existence of traveling pulse solutions and the asymptotic behavior
of solutions to the Cauchy problem (\ref{origequa1}) are determined by the sign
of the generalized principal eigenvalue $\lambda_1$, which is
defined by
\begin{equation}\label{mz}\lambda_1:=\sup\big{\{}\lambda|\exists ~\phi\in C^2(\mathbb{R}),~\phi(z)>0~{\rm and}~\phi''+c\phi'+f_u(z,0)\phi+\lambda\phi\leq 0~{\rm in}~\mathbb{R}\big{\}}.
\end{equation}
Precisely speaking, in \cite{[BDNZ]} it was shown that if $\lambda_1>0$ then all the solutions to the Cauchy problem of (\ref{origequa1}) with compactly supported initial data tend to zero
as $t\to+\infty$ (the species will die out); while if $\lambda_1<0$ then all
the positive solutions of (\ref{origequa1}) tend to the unique traveling pulse solution $\phi_c(x-ct)$ (the species will survive).
Similar results for the related models
in high dimensional space were also obtained in \cite{[BR2]} and \cite{[BR22]}.

For model (\ref{origequa1}) with a shifting term $f(x-ct,u)$ satisfying $f_u(-\infty,0)>0$,
Berestycki et al. \cite{[BR1],[BHR]} introduced another notion
of generalized principal eigenvalue $\lambda'_1$,  which is defined by
\begin{equation}\label{mz2}\lambda'_1:=\inf\big{\{}\lambda|\exists~ \phi\in C^2(\mathbb{R})\cap W^{2,\infty}(\mathbb{R}),~s.t.,~\phi''+c\phi'+f_u(z,0)\phi+\lambda\phi\geq0~{\rm in}~\mathbb{R}\big{\}}.
\end{equation}
Applying their results to (\ref{origequa1}),
it follows that (\ref{origequa1}) admits at least one forced traveling wave $\phi_c(x-ct)$
if $\lambda'_1<0$ and no wave solution if $\lambda'_1>0$.
Some related recent works for model (\ref{origequa1}) can be found in  \cite{[BR3],[FLW]}.
In \cite{[FLW]}, the authors derived the following explicit formula for $\lambda_1$ and $\lambda'_1$ (defined by (\ref{mz}) and (\ref{mz2})):
if $f_u(-\infty,0)>0\geq f_u(+\infty,0)$ and $\frac{\partial}{\partial z}f_u(z,0)<0$, then
\begin{equation}\label{S}
\lambda_1(c)=-f_u(-\infty,0)+\frac{c^2}{4},~~~~
\lambda'_1(c)=
\begin{cases}
-f_u(-\infty,0),~~~c\leq0,\\
\lambda_1(c),~~~~~~~~~~~~c\in(0,\overline{c}),\\
-f_u(+\infty,0),~~~~c\geq\overline{c},
\end{cases}
\end{equation}
with $\overline{c}=2\sqrt{f_u(-\infty,0)-f_u(+\infty,0)}$.


For the case when  $f(z,u)$ is non-increasing in $z$, $f(z,0)=f(z,1)=0$ for any $z\in \mathbb{R}$ and both $f_\pm(u)=f(\pm\infty,u)$ are   ignition, bistable or non-degenerate Fisher type, Hamel (\cite{[F1],[F]}) investigated the existence and asymptotic behavior of a cylinder  wave $U_c(x-ct,y)$  satisfying the following heterogeneous PDE equation in a cylinder  $\Sigma=\{(z,y),~z\in \mathbb{R},~y\in\omega\subseteq \mathbb{R}^{N-1}\}$ under  the zero Neumann boundary condition
\begin{equation}\label{vt}
\begin{cases}
\Delta_{z,y} U+\beta(z,y,c) \partial_z U+f(z,U)=0,~(z,y)\in\Sigma,\\
\partial_\nu U=0,~(z,y)\in \partial\Sigma,\\
U(-\infty,y)=1,U(+\infty, y)=0,
\end{cases}
\end{equation}
such $f(z,u)$ can be in the form $f(z,u)=b(z)\widetilde{f}(u)$ with $b(z)$ having positive limits at $z=\pm\infty$, or $f(z,u)=\widetilde{f}(u)+g(z,u)$ with a small perturbation
$g(z,u)\rightarrow 0$ as $z\to\pm\infty$.
Especially when $f(\pm\infty,u)$ are  non-degenerate asymptotic Fisher type, i.e. $f_u(\pm\infty,0)>0$,  $f_u(\pm\infty,1)<0$, under the assumption that $\beta(z,y,c)$ is  non-increasing in $z$; in \cite{[F]} it was proved that (\ref{vt}) admits at least one cylinder wave solution $U_c(z,y)$  if $c\geq c_-^*:=2\sqrt{f_u(-\infty,0)}$, and there exists no such wave solution if $c<c_+^*:=2\sqrt{f_u(+\infty,0)}$.

Berestycki and Fang \cite{[BF]} investigated the model (\ref{origequa1}) in one dimensional space, for the case  when  $f(z,0)=0$ for $z\in \mathbb{R}$ and $f(z,u)$ is asymptotic Fisher type as $z\to -\infty$ (i.e. $f(-\infty,u)$ has a unique positive root $\alpha$ and $f_u(-\infty,\alpha)<0$), and satisfies   $\frac{f(z,u)}{u}$ is non-increasing in $u>0$ for any $z\in \mathbb{R}$.
They obtained nearly complete results on existence, multiplicity and ordering structure of all forced traveling waves $\phi_c(x-ct)$ connecting $\phi_c(-\infty)=\alpha>0$ and $\phi_c(+\infty)=0$ of (\ref{origequa1}) except some degenerate cases when $f_u(+\infty,0)=0$.
In \cite{[BF]} it was shown that
(\ref{origequa1}) admits a unique forced wave solution when $f(+\infty,u)<0,~u\in(0,1)$ for $c<c_-^*:=2\sqrt{f_u(-\infty,0)}$;
 and there are infinitely many forced wave solutions
for $c>c_+^*:=2\sqrt{f_u(+\infty,0)}$ when both $f(\pm\infty,u)$
are Fisher type covering some typical cases such as $f(z,u)=u(a(z)-u)$ with $a(\pm\infty)>0$ and $f(z,u)=b(z)\widetilde{f}(u)$ with $b(\pm\infty)>0$.

There are many interesting research works on the existence of forced waves for
the model (\ref{origequa1}) with $f(z,u)=u(a(z)-u)$ satisfying $a(-\infty)=\alpha>0$,
$a(+\infty)=\beta<0$ and $a'(z)\leq 0,~z\in\mathbb{R}$, i.e
\begin{equation}\label{Forceddd11}
u_t=u_{xx}+u\big{(}a(x-ct)-u\big{)},~~~t>0,~x\in \mathbb{R},
\end{equation}
in such case of $a(z)$ the model (\ref{Forceddd11}) means
the spatial habitat is divided into a favorable region and a  unfavorable region.
Hu and Zou \cite{[HZ]} obtained the existence of monotone forced wave solution of (\ref{Forceddd11}) connecting $\phi_c(-\infty)=\alpha$ and $\phi_c(+\infty)=0$ for any $c<0$.
Fang et al.\cite{[FLW]} and Chen et al.\cite{[CTW]} proved that there exists a unique monotone decreasing forced wave connecting $\phi_c(-\infty)=\alpha$ and $\phi_c(+\infty)=0$ for any given $c\in(0,2\sqrt{\alpha})$.
Fang et al.\cite{[FLW]} further proved that there are infinitely many traveling pulse solutions satisfying $\phi_c(-\infty)=\phi_c(+\infty)=0$ for $c\leq -2\sqrt{\alpha}$.

In \cite{[BG],[HL],[LBSW],[VHH]}, the authors
studied the persistence and spreading dynamics of solutions for (\ref{origequa1}) or (\ref{Forceddd11}).
Other related population models with local or nonlocal  diffusion have been deeply studied in \cite{[Guo],[LY],[LWZ],[YCW],[QLW],[YZ],[ZK]}, see also \cite{[WLDQ]} for a recent survey of the related works on the ecological models with shifting environment.

As far as we know, there are very  few  literatures  concerning the  propagating dynamics and forced wave for the model \eqref{origequa1} with general asymptotic Fisher type $f(z,u)$ at $z=\pm\infty$ but $f(z,u)$ degenerates at $z=+\infty$ and $u=0$, i.e. $f(-\infty,\alpha)=0$ for some $\alpha>0$, $f(z,0)=0$ for any $z\in \mathbb{R}$ and $f_u(+\infty,0)=0$.  Even for the typical  degenerate  case when $f(z,u)=u(a(z)-u)$ with $a'(z)\le 0$ for all $z\in \mathbb{R}$ or $\frac{\partial}{\partial z}f_u(z,0)<0$ for all $z\in \mathbb{R}$, the formula   \eqref{S} obtained in  \cite{[FLW]} indicates that $\lambda_1'(c) = 0$ for all $c \geq 2\sqrt{f_u(-\infty,0)}$, which corresponds to the critical case when $\lambda_1'(c) = 0$. In  \cite{[BR1]} the authors raised an theoretical open problem about the existence of positive wave for more general nonlinear heterogeneous equation in such critical case (see Problem 4.6 in  \cite{[BR1]}: Does $u''(z)+cu'+f(z,u)=0$ admit positive bounded solutions if $\lambda'_1(c) = 0?$). The multiplicity/uniqueness  and the asymptotic stability  of traveling waves for the model \eqref{origequa1} with heterogeneous  degenerate $f(z,u)$ are also the interesting and challenging problems in both math and application.

In the following of this paper, we focus on the Fisher-KPP equation \eqref{Forceddd11} with $c>0$ and a degenerate moving resource function $a(x-ct)$  satisfying $a(-\infty)=\alpha>0$~and $a(+\infty)=0$ (asymptotic degenerate  Fisher type at $z=+\infty$ and $u=0$),
which is also a typical example of (\ref{origequa1}) for the degenerate case $f_u(+\infty,0)=0$.  From the biological point of view, the decay rate of the resource function $a(z)$ near $z=+\infty$  may characterize the sharpness of the shifting habitat edge: exponential decay to zero corresponds to a distinct and steep habitat boundary with rapid resource depletion; non-exponential slow decay reflects gradual habitat degradation
and slow resource exhaustion.

In \cite{[BF]},
the authors investigated  the Fisher-KPP model (\ref{Forceddd11}) for two sub-cases  of such degenerate $a(z)$ with $c>0$ and  proved  that the existence and the  asymptotic behavior of forced waves of (\ref{Forceddd11}) in such degenerate cases heavily depend on the decay rate of $a(z)$ near $z=+\infty$. Precisely speaking, they showed  that (\ref{Forceddd11}) admits a unique forced wave solution decaying exponentially at $z=+\infty$ if
$a(z)=o(e^{-\gamma z})$ for $\gamma>c$ and $c\in (0,2\sqrt{\alpha})$; and has a continuum of forced waves decaying slower than some algebraic rate when $a(z)=O(z^{-\gamma})$ for $\gamma\in(0,1)$ and any $c>0$.

In what follows, we always assume that $c>0$ and $a(z)$ satisfies
\begin{enumerate}
\item[({\bf $H_1$})]
$a(z)\in (0,\alpha]$ is a positive and continuous
function in whole $\mathbb{R}$ and is $C^1$ near $z=+\infty$ satisfying
$
a(+\infty)=0<a(-\infty)=\alpha,\;{\rm and}\;\; a'(z)\le 0,\;z\gg 1.
$
\end{enumerate}
Our main research interest  is to classify $a(z)$ by the decay rate of $a(z)$ at $z=+\infty$
and  investigate the decaying effect of $a(z)$ on
the existence, the multiplicity and the precise spatial decay of all forced waves. Due to the  degeneracy of both the linear and nonlinear reaction terms at $z=+\infty$ and $u=0$, (\ref{Forceddd11}) may admit a unique or infinitely many or multi-types of forced waves, and the spatial decay of some waves
may be not dominated by the  linearized heterogeneous  equation of \eqref{Forceddd11} at $u=0$ near $z=+\infty$ when  $\int_{z_0}^{+\infty}a(z){\rm d}z=+\infty$.
 It is worth mentioning that asymptotic analysis on these waves are more complicated than those for non-degenerate case of $a(z)$.

In Section 2, we shall sketch the strategy of our proofs and the approaches applied in our asymptotic analysis on some types of forced waves; and  shall  state the main results on the existence, the uniqueness, the multiplicity and the asymptotic behavior of all the forced waves
for  all the degenerate cases of $a(z)$ when $a(z)$ satisfies ({\bf $H_1$}).

\section{Formulation of the problems and statement of main results}

Let $u(t,x)=\phi_c(x-ct)$ be a forced traveling wave of (\ref{Forceddd11}),
then $\phi_c(z)$ is a positive global solution to the following
differential equation:
\begin{equation}\label{questt1}
\phi''_{c}(z)+c\phi'_c(z)+\phi_c(z)(a(z)-\phi_c(z))=0,~z\in \mathbb{R}
\end{equation}
with some asymptotic boundary conditions at $z=\pm\infty$.

\begin{lemma}\label{negnega1}
Assume that $a(z)$ satisfies ({\bf $H_1$}) with $c>0$,
let $\phi_c(z)$ be a positive global solution of (\ref{questt1}). Then
\begin{equation}\label{boudc1}
\begin{array}{c}0<\phi_c(z)\le \alpha,~\forall z\in\mathbb{R},\;\; \phi'_c(z)<0~~ {\rm for}~z\gg 1;\\
{\rm and}\quad \lim\limits_{z\to-\infty}\phi_c(z)=\alpha,~~
\lim\limits_{z\to+\infty}\phi_c(z)=0.
\end{array}
\end{equation}
\end{lemma}
\begin{proof}
{\rm The proof of Lemma \ref{negnega1} is left in the Appendix.}
\end{proof}\\
\begin{remark}
The proof of Lemma \ref{negnega1} also guarantees that if $a'(z)\le 0 $ for whole $z\in \mathbb{R}$ then $\phi'_c(z)<0$ and $0<\phi_c(z)\le \alpha$ for whole $z\in \mathbb{R}$, which is a traveling front solution of (\ref{questt1}); however if $a(z)$ is  non-decreasing in some long range, $\phi_c(z)$ may be not a traveling front solution.
\end{remark}
It is worth mentioning that the more general results and the  argument about the ordering of all forced waves, the existence of the maximal forced wave and the criteria  for the existence/non-existence of  minimal forced wave obtained  in  \cite{[BF]} for the more general model \eqref{origequa1} are still valid for the model (\ref{Forceddd11}) with a degenerate $a(z)$, which will also be useful in our later investigation.  For convenience of later citation, we summarize and restate the related results  obtained in \cite{[BF]} under the additional assumption of ({\bf $H_1$}) as the following  Theorem A.

{\bf Theorem A}~{\cite[Theorems 1.1-1.2]{[BF]}}\label{pppthmm1}
Under the assumption ({\bf $H_1$}) with $c>0$,  let $\phi_c(z)$ be a positive and globally bounded solution satisfying
\begin{equation}\label{nonl}
\begin{array}{l}
\phi_c''(z)+c\phi_c'(z)+\phi_c(z)(a(z)-\phi_c(z))=0,~~z\in\mathbb{R},
\\ \phi_c(-\infty)=\alpha>0.
\end{array}
\end{equation}
Then (\ref{nonl}) has no such solution, or has a unique solution, or has a continuum of solutions which are ordered.\\
(i) If (\ref{nonl}) has more than one positive global solution, then the maximal global solution $\phi^{\max}_c(z)$ of (\ref{nonl}) exists, with $\phi^{\max}_c(z)$ defined by
 $$\phi^{\max}_c(z):=\sup\big{\{}\phi_c(z)\,|\,
\phi_c''(z)+c\phi_c'(z)+\phi_c(z)(a(z)-\phi_c(z))=0,~\phi_c(z)>0,~z\in\mathbb{R}\big{\}}. $$
(ii) (\ref{nonl}) has a minimal positive global solution if and only if $\lambda_1<0$, where $\lambda_1$ is defined as \eqref{mz}.

Before stating our main results on the spatial decay of all forced waves $\phi_c(z)$ of (\ref{questt1}) and (\ref{boudc1}) near $z=+\infty$,  we
first classify any given $a(z)$ satisfying $(H_1)$ by the spatial decay of $a(z)$ near $z=+\infty$,
and then sketch the strategy of our asymptotic analysis on all local positive solutions decaying at $z=+\infty$ and introduce some classical or more general ODE asymptotic theories  which are applicable  to some sub-cases of $a(z)$.

In the following, let $\psi(z)$ be a local positive solution near $z=+\infty$ satisfying
\begin{equation}\label{bz}
\psi^{\prime\prime}+c\psi^{\prime}+\psi\big{(}a(z)-\psi\big{)}=0,~~ \psi(z)>0,\;z>z_0;\;\;\;\psi(+\infty)=0,
\end{equation}
for some $z_0\in\mathbb{R}$.
In Section 3,
we shall investigate the existence and the spatial decay of each local positive solution $\psi(z)$ to the problem
(\ref{bz}) as $z\to+\infty$.
As proved in Lemma \ref{negnega1}, such solution must satisfy $\psi'(z)<0$ for $z\gg1$.
We rewrite
(\ref{bz}) as the following first order nonlinear system
\begin{equation}\label{for1}
\mathbf{X}'=
\big{(}\Theta(z)+B(\mathbf{X})\big{)}\mathbf{X},~~\psi(z)>0,~\psi'(z)<0,~{\rm for}~z\geq z_0\gg1;\; \mathbf{X}(+\infty)=\overrightarrow{0},
\end{equation}
with~$\mathbf{X}=\big{(}
\psi(z),~
\psi'(z)\big{)}^{\top}$ and
$$
\Theta(z)={\left[
\begin{array}{c}
~~0~~~~~~~~~~1\\
-a(z)~~~-c
\end{array}
\right ]},~~
B(\mathbf{X})={\left[
\begin{array}{c}
~~0~~~~~~~0\\
\psi(z)~~~~~0
\end{array}\right ]}.
$$
Notice that $\Theta(z)\to\Theta^+$ as $z\to+\infty$ with
$
\Theta^+={\left[
\begin{array}{c}
0~~~~~1\\
0~~-c
\end{array}
\right ]}.
$
Using $a(+\infty)=0$,
it is easy to see that $\Theta(z)$ has two eigenvalues $\sigma_j(z)(j=1,2)$ satisfying
\begin{equation}\label{eig1}
\begin{aligned}
&\sigma_1(z)=\frac{-c-\sqrt{c^2-4a(z)}}{2}=-c+\frac{1}{c}a(z)+O\big{(}a^2(z)\big{)}\le -\frac{c}{2},~~z\gg1,\\
&\sigma_2(z)=\frac{-c+\sqrt{c^2-4a(z)}}{2}=-\frac{a(z)}{c}-\frac{2 a^2(z)}{c^3}+o\big{(}a^2(z)\big{)},~~z\gg1.
\end{aligned}
\end{equation}

{\bf Case I.}
$\int_{z_0}^{+\infty}a(z){\rm d}z<+\infty$.
In this case by applying the classical ODE asymptotic theories, it is known that for each $c>0$ and $K>0$, the problem (\ref{bz}) has a unique local positive
solution $\psi(z)$ which decays exponentially as $z\to+\infty$ and satisfies
\begin{equation} \label{expdx}
\psi(z)\sim K{\rm e}^{-cz}~~{\rm as}~z\to+\infty.
\end{equation}
Furthermore, using the fact that any non-exponential decay solution of (\ref{bz}) must satisfy $\psi(z)<a(z)$ for $z\gg1$ (see Prop.3.1), which implies  $\int^{+\infty}_{z_0}\big(a(z)-\psi(z)\big){\rm d}z<+\infty$. Then by applying classical ODE asymptotic theories it is easy to prove that in the {\bf Case I} (\ref{bz}) has no local positive solution decaying non-exponentially, which also means that the decaying rates of all the local positive solutions to \eqref{bz} near $z=+\infty$ are dominated by the linear homogeneous  limiting equation  $\psi^{\prime\prime}(z)+c\psi^{\prime}(z)=0$.\\

{\bf Case II.}
$\int_{z_0}^{+\infty}a(z){\rm d}z=+\infty$.
In such case, it is well known that the asymptotic behavior of a solution
$(\psi(z),~\psi'(z))^{\top}$ to (\ref{for1}) near $z=+\infty$ is no longer dominated  by  the linear homogeneous limiting system $\mathbf{X}'=
\Theta^+\mathbf{X}$, which is more closely related with the decay rates of  the solutions to the linear heterogeneous system $\mathbf{X}'=
\Theta(z)\mathbf{X}$. Using the fact that $\exp\big{(}\int_{z_0}^{z}\sigma_1(s){\rm d}s\big{)}$ decays to zero exponentially  as $z\to+\infty$; it can be proved that for each fixed $K>0$
there exists a unique local positive solution $\psi(z)$ to the nonlinear equation (\ref{bz}) satisfying
\begin{equation}\label{ogf0i}\psi(z)\sim K\exp\Big{(}\int_{z_0}^z\sigma_1(s){\rm d}s\Big{)},~{\rm as}~z\to+\infty,\end{equation}
which decays exponentially fast but decay a little slower than $K_1{\rm e}^{-cz}$ for any $K_1>0$. Only note that
$-c<\sigma_1(z)<-c+\delta,~\forall~ 0<\delta\ll 1$.

Now we sketch the proof of the existence of such local positive and exponential  decay solution.
After some transformation
$\mathbf{X}=T(z)P(z)$,  (\ref{for1}) becomes the following nonlinear system of $P(z):$
\begin{equation}\label{eig111}
P'(z)=\Lambda(z)P+\{T^{-1}T'+\widetilde{B}(P,z)\}P,~~
P(z)={\left[
\begin{array}{c}
P_1(z)\\
P_2(z)
\end{array}
\right ]};~
P(+\infty)=\overrightarrow{0}, \end{equation}
where $\Lambda(z)={\rm diag}\{\sigma_1(z),\sigma_2(z)\}$ and $\widetilde{B}(P,z)=O(|P(z)|)$ if $|P(z)|\ll 1$, and
$$T(z)={\left[
\begin{array}{c}
1~~~~~~~~~~~1\\
\sigma_1(z)~~~~\sigma_2(z)
\end{array}
\right ]},~~T^{-1}(z)T'(z)=\frac{1}{\sigma_+(z)-\sigma_-(z)}{\left[
\begin{array}{c}
-\sigma'_1(z)~~~-\sigma'_2(z)\\
-\sigma'_1(z)~~~~~~\sigma'_2(z)
\end{array}
\right ]}. $$
By virtue of $a'(z)\leq 0$~for $z\gg 1$,
it is easy to see that for any fixed $z_0\gg1$
\begin{equation*}\label{intecon1}
\int_{z_0}^{+\infty}|\Lambda'(s)|{\rm d}s<+\infty~{\rm and}~\int_{z_0}^{+\infty}\big{|}T^{-1}(s)T'(s)\big{|}{\rm d}s<+\infty.
\end{equation*}
By applying more general ODE asymptotic theories ({see \cite[Chapter~IV,~Theorem 11]{[CWA]}}) to \eqref{bz} or (\ref{eig111}),
it follows that if a local positive solution $\psi(z)$ to \eqref{eig111} satisfies $\int^{+\infty}_{z_0} |P(z)|{\rm d}z<+\infty$ (or a local positive  solution $\psi(z)$ to \eqref{bz} satisfies $\int^{+\infty}_{z_0} \psi(z){\rm d}z<+\infty$), then the decay rate of such solution to \eqref{eig111} is
the same as one of solution to the linear heterogeneous system $P'(z)=\Lambda(z)P(z)$. Further, by applying the standard perturbation argument it can be proved that for each $K>0$ the nonlinear equation \eqref{bz} with $c>0$
admits a unique local positive solution decaying exponentially  as $z\to +\infty$ and satisfying (\ref{ogf0i}).

 Note that the assumption $\int_{z_0}^{+\infty}a(z){\rm d}z=+\infty$ and (\ref{eig1}) imply that
 $\int_{z_0}^{+\infty}\sigma_2(z){\rm d}z=-\infty$, thus the linear heterogeneous system $P'(z)=\Lambda(z)P(z)$ (or the linear heterogeneous equation $\psi^{\prime\prime}(z)+c\psi^{\prime}(z)+a(z)\psi(z)=0$) also admits a family of  positive solutions decaying like  $K\exp\big{(}\int_{z_0}^z\sigma_2(s){\rm d}s\big{)}$ with some $K>0$, which decay to zero non-exponentially (decay slower than
any exponential decay function) as $z\to +\infty$. Moreover, if $\exp\big{(}\int_{z_0}^z\sigma_2(s){\rm d}s\big{)}\in L^1([z_0,+\infty))$, then by applying the standard perturbation argument it can be proved that the nonlinear equation \eqref{bz} has a  family of local positive solutions $\psi(z)$  decaying like $K\exp\big{(}\int_{z_0}^{z}\sigma_2(s){\rm d}s\big{)}$ for some constant $K>0$.

In this paper we are more interested in the case when $\int_{z_0}^{+\infty}a(z){\rm d}z=+\infty$, in such case \eqref{bz} may admit a family or several families of local positive solutions decaying non-exponentially at $z=+\infty$, which implies that there may exist multi-types of forced waves. The results on the existence and the precise non-exponential decay rates of local positive solutions of \eqref{bz} will be crucial in our later investigation on the existence and the multiplicity of some types of forced waves decaying non-exponentially. Moreover, our on-going research work also indicates that the nonlinear stability of forced waves and the longtime behavior of solutions with more general initial data also heavily depend on the precise decay rates of all the waves near $z=+\infty$.

In Section 3 we shall focus on the investigation of the existence/nonexistence and the precise non-exponential decay rate of all the local positive solutions of \eqref{bz}, Our later asymptotic analysis on non-exponential decay highly  depends on the decay rate of $a(z)$ and the existence/non-existence of such solution is  crucially determined by a criteria $\widetilde{\sigma}_2(z):=\exp\big{(}\int_{z_0}^z\sigma_2(s){\rm d}s\big{)}\in L^1([z_0,+\infty))$ or not, with
$\sigma_2(s)=\frac{-c+\sqrt{c^2-4a(z)}}{2}$.


In Section 3 we shall classify the decay rate of $a(z)$  by the following  three sub-cases:
\begin{equation*}
\begin{aligned}
&(1)\int_{z_0}^{+\infty}a(z){\rm d}z<+\infty;\\
&(2) \int_{z_0}^{+\infty}a(z){\rm d}z=+\infty\;{\rm and}\;\;\int^{+\infty}_{z_0}\widetilde{\sigma}_2(z){\rm d}z=+\infty;\\
&(3)\int_{z_0}^{+\infty}a(z){\rm d}z=+\infty \;{\rm and}\; \int^{+\infty}_{z_0}\widetilde{\sigma}_2(z){\rm d}z<+\infty.
\end{aligned}
\end{equation*}
which include the following  four typical decaying cases:
\begin{equation*}
\begin{aligned}
&{\rm (A)}
\varlimsup\limits_{z\to+\infty}za(z)<c, \;\;\;\;\;\;\;\;\;\;\;\;\;\;\;{\rm (B)}\;\lim\limits_{z\to+\infty}za(z)=c,\\
&{\rm (C)}
~c<\varliminf\limits_{z\to+\infty}za(z)< +\infty, \;\;\;\;\;\;\;{\rm (D)}
~\varliminf\limits_{z\to+\infty}za(z)= +\infty
\end{aligned}
\end{equation*}

In Section 3 by applying detailed asymptotic estimates and by applying contradictory argument we shall show that  for  Case (1) and Case  (2) the problem  \eqref{bz} has no local  positive solution decaying non-exponentially at $z=+\infty$, where  Case (1) and Case (2)  include  Case (A) and some  sub-critical cases  of Case (B).

In the Case (3), using the fact that $\widetilde{\sigma}_2(z)\in L^1([z_0,+\infty))$, by applying
 more general ODE asymptotic theories as stated above, it follows that  there exist a family of local positive solutions to the nonlinear equation \eqref{bz}, which decay like $K\widetilde{\sigma}_2(z)$ for some positive constant $K$.

For  Case (C): $c<\varliminf\limits_{z\to+\infty}za(z)< +\infty$, it is easy to  see  that
\begin{equation}\label{x}
\widetilde{\sigma}_2(z)\sim K\exp\Big(\frac{-1}{c}\int_{z_0}^za(s){\rm d}s\Big)=O\big{(}z^{-(1+\delta)}\big), \;\exists~\delta>0,\;{\rm as}~z\to+\infty.
\end{equation}
Thus $\widetilde{\sigma}_2(z)\in L^1([z_0,+\infty))$, which  guarantees that   there exist a family of local positive solutions to the nonlinear equation  \eqref{bz} decaying algebraically like $K\exp\Big(\frac{-1}{c}\int_{z_0}^za(s){\rm d}s\Big)$.

For the Case (D), it is easy to check that
 $\widetilde{\sigma}_2(z)\in L^1([z_0,+\infty))$ and $\widetilde{\sigma}_2(z)$ decays faster than any algebraic decaying function, thus \eqref{bz} admits a family of local positive solutions decaying like $K\widetilde{\sigma}_2(z)$, which decay non-exponentially but not algebraically.

For the critical case (B): $a(z)\sim \frac{c}{z}$, or more general critical  case of $a(z)$ (see Remark \ref{rem3.1}); the estimates on the decay rates of $\widetilde{\sigma}_2(z)$ seem to be more complicated than other three cases (non-critical), which highly depend on the second order or higher order asymptotic
expansion of $a(z)$ or the precise expression of $a(z)$ near $z= +\infty$. In Section 3 we classify the typical critical case $a(z)\sim \frac{c}{z}$ into several typical sub/super-critical cases (see \eqref{crit-} and \eqref{crit+}),  or simply classified by the criteria:  $\int_{z_0}^{+\infty}
\widetilde{\sigma}_2(z){\rm d}z$ is infinite or finite.  For the sub-critical case when
$\int_{z_0}^{+\infty}\widetilde{\sigma}_2(z){\rm d}z=+\infty $, our asymptotic results  obtained in Section 3 indicate that
  the nonlinear equation \eqref{bz} does not admit  non-exponential decaying local positive solution; however in the super-critical cases when   $\int_{z_0}^{+\infty}\widetilde{\sigma}_2(z){\rm d}z<+\infty $, by applying the above mentioned ODE asymptotic theories it follows that
 there exist a family of  local positive solution decaying  like $K\widetilde{\sigma}_2(z)$.

It is worth mentioning that for the Case (3), 
except the existence of infinitely many  non-exponential local positive solutions decaying like $K\widetilde{\sigma}_2(z)$, which are in  $L^1([z_0,+\infty))$; there may exist other types of local positive solutions which are not in  $L^1([z_0,+\infty))$. The existence and the precise decay of such local positive solutions with slower non-exponential decay to the nonlinear equation \eqref{bz} are closely related with both the linear term $a(z)\psi(z)$ and  nonlinear term $-\psi^2(z)$ in the equation \eqref{bz}. In Section 3 we shall also prove that for the Case (3), \eqref{bz} also admits another type of local positive solutions which decay non-exponentially and  are not in $L^1([z_0,+\infty))$, and thus decay more slowly than
$\widetilde{\sigma}_2(z)$, we shall also obtain the precise decay rate of such local positive solutions.

In Sections 4, we shall
 further prove the existence and the uniqueness of non-exponential traveling waves with the  slowest decaying rate, which corresponds to the maximal wave of (\ref{Forceddd11}) for Case (3).

If we consider the spatial decay of a local positive solution $\psi(z)$ of (\ref{questt1}) for $z\ll-1$ satisfying $\psi(-\infty)=
\alpha>0$. Let $v=\alpha-\psi(z)$, then $v(z)$ satisfies
\begin{equation}\label{cmbb}
v''+cv'+\big{(}a(z) -2\alpha\big{)v+v^2}-\alpha\big{(}a(z)-\alpha\big{)}=0,~~z\ll-1,
\end{equation}
and $v(-\infty)=0$. Due to the non-degeneracy of $f(z,u)=u(a(z)-u)$ at $z=-\infty$ and $u(-\infty)=\alpha>0$, it is naturally expected that the decay rate of a solution $v(z)$ of (\ref{cmbb})
at $z=-\infty$ is closely related with the following inhomogeneous linear   equation:
\begin{equation}\label{cmbbw2}
 v''+cv'+\big{(}a(z) -2\alpha\big{)} v-\alpha\big{(}a(z)-\alpha\big{)}=0.
\end{equation}
 It can be proved that $\psi(z)\to \alpha$ exponentially as $z\to-\infty$  as long as $a(z)\to \alpha$ exponentially as $z\to-\infty$; and $\psi(z)\to \alpha$ non-exponentially as $z\to-\infty$  as long as $a(z)\to \alpha$ non-exponentially as $z\to-\infty$. It is worth mentioning that
the spatial decay of $a(z)$ at $z=-\infty$ does not affect our later investigation on the existence  or the multiplicity of forced waves of (\ref{questt1}).

Our main results in this paper can be stated as follows:

\begin{thm}\label{thmfirst1}
Suppose that  ({\bf $H_1$}) is satisfied and let  $c>0$ be fixed.

{\bf (I)} If $c\in (0,2\sqrt{\alpha})$, then (\ref{Forceddd11}) admits a
forced wave solution
$\phi_c(x-ct)$ satisfying $\phi_c(-\infty)=\alpha>0$ and $\phi_c(+\infty)=0$, which  decays exponentially at $z=+\infty$ and satisfies
\begin{equation}\label{o}
\phi_c(z)\sim A_1\exp\Big{(}\int_{z_0}^z\frac{-c-\sqrt{c^2-4a(s)}}{2}{\rm d}s\Big{)},~as~z\rightarrow+\infty,
\end{equation}
for some $z_0\in\mathbb{R}$ and $A_1>0$. Moreover, the forced wave $\phi_c(z)$ decaying exponentially  at $z=+\infty$ must be unique.

{\bf (II)} If $c\ge 2\sqrt{\alpha}$, then  (\ref{Forceddd11}) has no forced wave decaying exponentially at $z=+\infty$.

{\bf (III)} If $c>0$ and  $a(z)$
 further satisfies
\begin{equation}\label{214}
  \int_{z_0}^{+\infty}\widetilde{\sigma}_2(z){\rm d}z=+\infty,\;\widetilde{\sigma}_2(z)=\exp\left(\int^z_{z_0} \frac{-c+\sqrt{c^2-4a(s)}}{2}{\rm d}s\right),\end{equation}
for some $z_0\in \mathbb{R}$, then (\ref{Forceddd11}) has no forced wave $\phi_c(z)$ decaying non-exponentially at $z=+\infty$.
 \end{thm}


\begin{remark}\label{rem2.1}
  The assumption (\ref{214}) includes the following three
typical cases:
\begin{enumerate}
\item[${\rm (i)}$]
$\displaystyle\int_{z_0}^{+\infty}a(z){\rm d}z<+\infty$;
\end{enumerate}
\begin{enumerate}
\item[${\rm (ii)}$]
$\displaystyle\int_{z_0}^{+\infty}a(z){\rm d}z=+\infty$ and $\varlimsup\limits_{z\rightarrow+\infty}za(z)< c$;
\end{enumerate}

\begin{enumerate}
\item[${\rm (iii)}$]
$\lim\limits_{z\rightarrow+\infty}za(z)=c$ and there exist an integer $k\geq1$ and $0<r\leq c$
such that
\begin{equation}\label{crit-}
\begin{aligned}
a(z)\leq& c\sum_{j=0}^{k-1}\Big{(}\frac{1}{z}\prod_{i=0}^j\big{(}\ln^i z\big{)}^{-1}\Big{)}
+\frac{r}{z}\prod_{i=1}^k\big{(}\ln^i z\big{)}^{-1},~~~z\gg1,
\end{aligned}
\end{equation}
with $\ln^{0}z=1$ and $\ln^jz:=\underbrace{\ln\cdot\cdot\cdot\ln}_jz$.
\end{enumerate}
 \end{remark}
The verification of Remark \ref{rem2.1} can be found in Section 3.

\begin{remark}\label{rem2.2}
Theorem  \ref{thmfirst1} also implies  that under the assumption of  (\ref{214}) if $c\in (0,2\sqrt{\alpha})$
then (\ref{Forceddd11}) admits a unique forced wave which must decay exponentially, however for any $c\ge 2\sqrt{\alpha}$ (\ref{Forceddd11}) has no forced wave.
\end{remark}

\begin{thm}\label{thmfirspk}
Suppose that  ({\bf $H_1$}) is satisfied and let $c>0$ be fixed, if $a(z)$ satisfies
\begin{equation}\label{plh}
\displaystyle\int_{z_0}^{+\infty}\widetilde{\sigma}_2(z){\rm d}z<+\infty,\;\;\widetilde{\sigma}_2(z)=\exp\left(\int^z_{z_0} \frac{-c+\sqrt{c^2-4a(s)}}{2}{\rm d}s\right),
\end{equation}
for some $z_0\in\mathbb{R}$, then (\ref{Forceddd11})
admits infinitely many ordered forced wave solutions $\phi_c(x-ct)$ satisfying $\phi_c(-\infty)=\alpha>0$
and $\phi_c(+\infty)=0$, which decay non-exponentially at $z=+\infty$ (decay slower than any exponential decay functions) and satisfy
\begin{equation}\label{fro}\phi_c(z)\sim K\widetilde{\sigma}_2(z)\in L^1([z_0,+\infty)),
~as~z\to+\infty,
\end{equation}
for some positive constant $K$ and there exist no other types of non-exponential
waves belonging in $L^1([z_0,+\infty))$.  Furthermore,
there exists a unique  forced wave not in $L^1([z_0,+\infty))$, which corresponds to  the maximal forced wave $\phi^{\max}_c(z)$ and must satisfy
\begin{equation}\label{2eqpb}
\phi^{\max}_c(z)\sim \frac{c\widetilde{\sigma}_2(z)}{\int_z^{+\infty} \widetilde{\sigma}_2(s) ds},~as~z\to+\infty.
\end{equation}
\end{thm}
%

\begin{remark}\label{rem2.3}
The assumption (\ref{2eqpb}) includes the following three typical decaying cases
of $a(z)$ and the following decaying estimates hold true.
\begin{enumerate}
\item[${\rm (i)}$] If
$\lim\limits_{z\rightarrow+\infty}za(z)=c$ and
there exist an integer $k\geq1$ and a positive constant $r>c$
such that
\begin{equation}\label{crit+}
\begin{aligned}
a(z)\geq
&c\sum_{j=0}^{k-1}\Big{(}\frac{1}{z}\prod_{i=0}^j\big{(}\ln^i z\big{)}^{-1}\Big{)}
+\frac{r}{z}\prod_{i=1}^k\big{(}\ln^i z\big{)}^{-1},
~~z\gg1,
\end{aligned}
\end{equation}
with  $\ln^{0}z=1$ and $\ln^jz:=\underbrace{\ln\cdot\cdot\cdot\ln}_jz$, then
$$\widetilde{\sigma}_2(z)\sim K\exp\Big(-\frac{1}{c}\int^{z}_{z_0}a(s){\rm d}s\Big)=O\Big(\frac{1}{z\ln z\cdot\cdot\cdot(\ln^kz)^\frac{r}{c}}\Big),\; as \;z\to+\infty;
$$
and
\begin{equation}\label{ouy32v}
\frac{c\widetilde{\sigma}_2(z)}{\int_z^{+\infty} \widetilde{\sigma}_2(s) ds}\geq (r-c)\big(z\ln z\cdot\cdot\cdot(\ln^kz)\big)^{-1},~as~z\to+\infty.
\end{equation}
\end{enumerate}

\begin{enumerate}
\item[${\rm (ii)}$] If
$\varliminf\limits_{z\to+\infty}za(z)=\gamma\in (c,+\infty)$, then
$\widetilde{\sigma}_2(z)=O(z^{-\frac{\gamma}{c}})$ for $z\gg 1$; and
$$ \frac{c\widetilde{\sigma}_2(z)}{\int_z^{+\infty} \widetilde{\sigma}_2(s) ds}\ge \frac{\gamma-c}{z},\; z\gg 1.
$$
\end{enumerate}

\begin{enumerate}
\item[${\rm (iii)}$] If
$\lim\limits_{z\to+\infty}za(z)=+\infty$,
then $\widetilde{\sigma}_2(z)=O\big(z^{-p}\big),~as~ z\to+\infty$, for any $p>1$ and $\lim\limits_{z\to+\infty}z\phi^{\max}_c(z)=+\infty$. Furthermore, if $\lim\limits_{z\to +\infty}\frac{a'(z)}{a^2(z)}=0$, then
$$  \frac{c\widetilde{\sigma}_2(z)}{\int_{z}^{+\infty}\widetilde{\sigma}_2(s){\rm d}s}\sim a(z), \;\;as\; z\to +\infty.$$
\end{enumerate}
\end{remark}
The verification of the estimates stated in Remark \ref{rem2.1} can be found in Section 3.

\begin{remark}
Theorems  \ref{thmfirst1} and \ref{thmfirspk} guarantee  that if  $~\widetilde{\sigma}_2(z)\in L^1([z_0,+\infty))$, then for any $c\in (0,2\sqrt{\alpha})$
 (\ref{Forceddd11}) admits infinitely many ordered forced waves with only one forced wave decaying  exponentially which is also the minimal forced wave, and for any $c>0$ and if $~\widetilde{\sigma}_2(z)\in L^1([z_0,+\infty))$, (\ref{Forceddd11}) admits infinitely many forced waves  decay non-exponentially at $z=+\infty$ and there are  two types  of non-exponential decaying rates of the forced waves. For $c>2\sqrt{\alpha}$ and
 $\widetilde{\sigma}_2(z)\in L^1([z_0,+\infty))$, (\ref{Forceddd11}) has no minimal forced wave and all the forced waves decaying non-exponentially.  Under the additional assumption of (\ref{plh}) only the maximal forced wave $\phi^{max}_c(z)$  decays with the slowest decaying rate (\ref{2eqpb}), all other non-exponential  forced waves  decay like $K\widetilde{\sigma}_2(z)$.
\end{remark}

 Biological interpretation:  The results sated in Theorems 2.1-2.2 indicate that the critical speed $2\sqrt{\alpha}$ governs the existence of exponentially decaying waves; while the convergence or divergence of the integral $\int_{z_0}^{+\infty}\widetilde{\sigma}_2(z){\rm d}z$, which reflects the rate at which $a(z)$ decays to zero, determines the existence of infinitely many non-exponentially decaying waves.

When $c<2\sqrt{\alpha}$ and the resource decays sufficiently fast, the system admits only a unique exponentially decaying forced wave. In this scenario, population density declines sharply outside the core resource region, forming a distinct and sharp distribution boundary. This uniqueness indicates that under predictable environmental change with a well-defined resource boundary, the population evolves a unique optimal migration strategy, which leads to a clear distribution front (e.g., the fixed routes of seasonally migrating species).

If $\int_{z_0}^{+\infty}\widetilde{\sigma}_2(z){\rm d}z<+\infty$, then resource function $a(z)$ decays slowly, corresponding to gradual degradation at the habitat leading edge. In this case, regardless of
$c$, the species can form infinitely many forced waves through diverse spatial strategies. This highlights high ecological adaptability: even under rapid environmental change, as long as resources do not vanish abruptly, the population can sustain migration via flexible distribution patterns.\medskip

This paper is organized as follows. In Section 3, we investigate the existence and the precise  decay rates of all the local positive solutions
of (\ref{bz}). In Section 4, we prove the uniqueness
of the forced wave decaying exponentially and the wave not belonging in  $L^1([z_0,+\infty))$, and complete the proofs of theorems \ref{thmfirst1}-\ref{thmfirspk} by proving the existence/non-existence of several types of  forced waves decaying exponentially  or non-exponentially.

\section{Existence and spatial decay of local positive solutions of (\ref{bz}) }
In this section under the assumption of ({\bf $H_1$}) we aim to prove the existence/nonexistence and to obtain the precise decay rates of all the local positive solutions $\psi(z)$ of (\ref{bz}) at  $z=+\infty$, which play key roles in
our later investigation on the existence/nonexistence and the multiplicity as well as the precise decay rates of all the forced waves $\phi_c(x-ct)$
of (\ref{Forceddd11}).

As discussed in Section 2, under the assumption ({\bf $H_1$}) the decay rates of all the local positive solutions to the nonlinear equation (\ref{bz}) near $z=+\infty$ may have only one type or several types of  possibilities, which
are closely related with the decay rate of $a(z)$,   or precisely speaking, which are  determined by the criteria $\int^{+\infty}_{z_0}\widetilde{\sigma}_2(z){\rm d}z$ is finite/infinite.  By applying ODE asymptotic theories,  the  decay rates of some types (such as the exponential decay)  of  local positive solutions of (\ref{bz})  near $z=+\infty$ are the same as those of  the heterogeneous  linearized system of (\ref{bz}) around $\psi(z)=0$, i.e. $\psi''(z)+c\psi'(z)+a(z)\psi(z)=0$; and it was shown in Section 2 that for the case when $\int^{+\infty}_{z_0}a(z){\rm d}z<+\infty $ neither the heterogeneous  linearized system nor the nonlinear equation (\ref{bz}) has local positive solution decaying non-exponentially near $z=+\infty$. In this section we are more interested in the case when
$\int^{+\infty}_{z_0}a(z){\rm d}z=+\infty $, in such case it is known that the linearized equation has a family of positive solutions decaying non-exponentially and decay like $K\widetilde{\sigma}_2(z)$; however the existence and the non-exponential decay rates of local solutions of the nonlinear equation (\ref{bz}) may  highly depend on the decay rate of $\widetilde{\sigma}_2(z)$ and the effect of nonlinear term of (\ref{bz}), which will be investigated in details in this section.

We begin with establishing  the  existence and the precise decaying rates of a family of local positive solutions of $(\ref{bz})$ which decay exponentially at $z=+\infty$, which are stated as follows:
\begin{lemma}\label{decavr}
Assume that ({\bf $H_1$}) is satisfied for some $c>0$. Then for large enough $z_0\gg 1$ and any given $K_1>0$, (\ref{bz}) admits a unique local positive solution $\psi(z)\in
C^2([z_0,+\infty))$ decaying exponentially and satisfying
\begin{equation}\label{xcr}\psi(z)\sim K_1\exp\Big{(}\int_{z_0}^z\sigma_1(s){\rm d}s\Big{)},~as~z\rightarrow+\infty,
\end{equation}
with $\sigma_1(s)=\frac{-c-\sqrt{c^2-4a(s)}}{2}$, and all the local positive solutions decaying exponentially
must satisfy (\ref{xcr}) for some  $K_1>0$. \end{lemma}
\begin{proof}
If $\int_{z_0}^{+\infty}a(z){\rm d}z<+\infty$,
by applying more general ODE asymptotic theories (see {\cite[Chapter~IV,~Theorem 1]{[CWA]}}), it follows that (\ref{bz}) admits a family of local positive solutions
decaying exponentially and satisfying (\ref{expdx}), and
there exists no local positive solution decaying non-exponentially as $z\to+\infty$, in such case $\exp(\int_{z_0}^z\sigma_1(s){\rm d}s)\sim A{\rm e}^{-cz}$ as $z\to+\infty$ for some $A>0$, which proves (\ref{xcr}).

If $\int_{z_0}^{+\infty}a(z){\rm d}z=+\infty$,
using the fact that
$a(+\infty)=0$ and $a'(z)\leq0$ for $z>z_0\gg1$,  we have
$\int_{z_0}^{+\infty}|a'(z)|{\rm d}z<+\infty$,
then by applying more general ODE asymptotic theories ({see \cite[Chapter~IV,~Theorem 11]{[CWA]}}) (or see
the sketched proof  of (\ref{ogf0i}) in Section 2),
it follows that for each given $K_1>0$ the equation
(\ref{bz}) has a unique local positive solution $\psi(z)$ satisfying
\begin{equation*}\label{asyconvtuy}
\begin{aligned}
\psi(z)\sim K_1\exp\Big(\int_{z_0}^z\sigma_1(s){\rm d}s\Big)=o\big{(}{\rm e}^{-(c-\epsilon)z}\big{)},~{\rm as}~z\to+\infty,
\end{aligned}
\end{equation*}
for any given small $\epsilon>0$, where $\sigma_1(z)=\frac{-c-\sqrt{c^2-4a(z)}}{2}\in (-c+\frac{a(z)}{c}, -c+\epsilon) $ for $z\gg 1$,
and there exist no other types of local positive solutions decaying exponentially at $z=+\infty$. This completes the proof.
\end{proof}

For the case when $\int_{z_0}^{+\infty}a(z){\rm d}z=+\infty$, as discussed in Section 2
$(\ref{bz})$ may admit  local positive solutions  decaying to zero non-exponentially.
Before investigating the existence and all the possible non-exponential decays of  local positive solutions of (\ref{bz}), we first prove some preliminary estimates.
\begin{prop}\label{lepro}
Assume that ({\bf $H_1$}) is satisfied for some $c>0$,
if (\ref{bz}) admits a local positive solution $\psi(z)$ decaying non-exponentially at $z=+\infty$, then
\begin{equation}\label{decze}
 0<\psi(z)<a(z),\;\psi'(z)<0,\;for~z\gg 1;\; {\rm and}\;\; \lim_{z\rightarrow+\infty}\frac{\psi'(z)}{\psi(z)}=\lim_{z\to+\infty}\frac{\psi''(z)}{\psi'(z)}=0.
\end{equation}
{\rm The proof of Proposition \ref{lepro} is left in the Appendix B.}\\
\end{prop}

\begin{thm}\label{deca}
Assume that  ({\bf $H_1$}) is satisfied for some $c>0$.
If  $a(z)$ satisfies
\begin{equation}\label{asuijgf}
\displaystyle\int^{+\infty}_{z_0}\widetilde{\sigma}_2(z){\rm d}z=+\infty; \;\widetilde{\sigma}_2(z): = \exp\left(\int_{z_0}^z \frac{-c+\sqrt{c^2 - 4a(s)}}{2} \, {\rm d}s\right),~ z\ge z_0\gg 1,
\end{equation}
then (\ref{bz}) has no local positive solution
decaying non-exponentially at $z=+\infty$.
\end{thm}
\begin{proof}
{By contradiction, let $\psi(z)$ be a local classical positive solution of (\ref{bz}) defined on $[z_0,+\infty)$, which decays non-exponentially at $z = +\infty $. Note that $\psi(z)\in C^3([z_0,+\infty))$ for $z_0\gg 1$. By Proposition \ref{lepro}, we know that
\begin{equation}\label{eq:3.2}
0 < \psi(z) < a(z) \quad \text{and} \quad \psi'(z) < 0,~{\rm for}~ z \geq z_0,
\end{equation}
then by applying ODE asymptotic estimates as stated  in Section 2, $\psi(z)$ must satisfy
\begin{equation}\label{eq:3.3}
\begin{array}{l}
\psi(z) \sim K_0 \widetilde{\sigma}_{\psi}(z),\quad \text{as } z \to +\infty,\;\text{ for some constant} ~K_0>0;\\
{\rm with}~\widetilde{\sigma}_{\psi}(z):=\exp\left(\int_{z_0}^z \sigma_{\psi}(s) \, {\rm d} s\right),\;\sigma_\psi(z):= \frac{-c + \sqrt{c^2 - 4(a(z) - \psi(z))}}{2},\;z\ge z_0;
\end{array}
\end{equation}
which implies that $\psi(z)$ can be expressed as
\begin{equation}\label{eq:3.4}
\psi(z) = v(z) \widetilde{\sigma}_\psi(z), \quad \text{for } z \geq z_0\gg1,
\end{equation}
with a $C^3$ function $v(z)$ satisfying $v(z) \to K_0 > 0$  as $z \to +\infty$.
Note that
\begin{align}
\sigma_{\psi}(z)-\sigma_2(z) &= \frac{2\psi(z)}{\sqrt{c^2 - 4(a(z) - \psi(z))}+ \sqrt{c^2 - 4a(z)}},~ z \geq z_0, \notag \\
&= \frac{\psi(z)}{c + b(z)} \geq \frac{\psi(z)}{c}, \quad \text{for } z \geq z_0, \label{eq:3.6}
\end{align}
with $b(z) \to 0^- $  as $z \to +\infty $ .
By \eqref{eq:3.6}, \eqref{eq:3.4} can be rewritten as the following integral equation of $\psi(z)$
\begin{align}\label{eq:3.7}
\psi(z) = v(z) \widetilde{\sigma}_2(z) {\rm e}^{\int_{z_0}^z \frac{\psi(s)}{c + b(s)} {\rm d}s}
\geq v(z) \widetilde{\sigma}_2(z), \quad z \geq z_0,
\end{align}
with $v(z) \to K_0>0$  and $b(z) \to 0^- $  as $z \to +\infty$ ,
which with the assumption \eqref{asuijgf} ($\int_{z_0}^{+\infty}\widetilde{\sigma}_2(z){\rm d}z=+\infty)$ further implies
\begin{equation}\label{eq:la}
\int_{z_0}^{+\infty} \psi(z) {\rm d}z = +\infty.
\end{equation}
Let
\begin{equation}\label{eq:vy}
\psi_1(z) = \frac{\psi(z)}{c + b(z)}\;\;{\rm and}\;\;\Psi_1(z) = \int_{z_0}^z \psi_1(s) {\rm d}s,~z\geq z_0\gg1.
\end{equation}
Note that  \eqref{eq:la}, \eqref{eq:vy} and the fact $b(z) \to 0$  as $z \to +\infty $
imply $\psi_1(z) \notin L^1([z_0, +\infty))$,} thus
$\Psi_1(+ \infty) = +\infty $, and the nonlinear integral equation  \eqref{eq:3.7} can be rewritten as the following   nonlinear differential  equation of $\Psi_1(z)$:
\begin{equation}\label{eq:3.10}
\exp(-\Psi_1(z)) \Psi'_1(z) = \frac{v(z)}{c+b(z)} \widetilde{\sigma}_2(z), \quad z \geq z_0,
\end{equation}
with $b(z) \to0^-$ and  $v(z) \to K_0 > 0$  as $z \to +\infty$.

Integrating the equation \eqref{eq:3.10} from $z_0$  to $z $, we have
\begin{equation}\label{eq:3.11}
1 - \exp(-\Psi_1(z)) = \int_{z_0}^{z} \frac{v(s)}{c+b(s)} \widetilde{\sigma}_2(s) \, {\rm d}s, \quad z \geq z_0.
\end{equation}
Using $\Psi_1(+\infty)=+\infty$ and
letting  $z \to +\infty$ in \eqref{eq:3.11}, we have
\begin{equation}\label{eq:3.12}
1 = \int_{z_0}^{+\infty} \frac{v(s)}{c+b(s)} \widetilde{\sigma}_2(s) \, {\rm d}s,
\end{equation}
which  contradicts with the assumption
$\int_{z_0}^{+\infty} \widetilde{\sigma}_2(s) \, {\rm d}s = +\infty$;  noting that $\frac{v(z)}{c+b(s)}\to  \frac{K_0}{c}> 0$  as $z\to +\infty$.
\end{proof}

\begin{proposition}\label{prop3.1}
The assumption $(\ref{asuijgf})$ includes the three typical cases as stated in the Remark  \ref{rem2.1}.
\end{proposition}

\begin{proof}
Note that
\begin{equation}\label{eig9}
\sigma_2(z) = -\frac{a(z)}{c} - \frac{2a^2(z)}{c^3} + o(a(z)), \quad z \ge z_0 \gg 1,
\end{equation}

For the case (i): $\displaystyle\int_{z_0}^{+\infty}a(z){\rm d}z<+\infty$; it is easy to see that
$|\sigma_2(z)|<\frac{2}{c}a(z)$ for $z\gg1$, thus   $\widetilde{a}(z)=\exp\big(\frac{-1}{c}\int_{z_0}^za(s){\rm d}s\big)$, $\widetilde{\sigma}_2(z)\notin L^1([z_0,+\infty))$.

For the case (ii): $\displaystyle\int_{z_0}^{+\infty}a(z){\rm d}z=+\infty$ and $\varlimsup\limits_{z\rightarrow+\infty}za(z)< c$; under the assumption $a(z)\notin L^1([z_0,+\infty)$ and $\varlimsup\limits_{z\rightarrow+\infty}za(z)< c$, by (\ref{eig9}), obviously,
$\widetilde{\sigma}_2(z)\sim K\widetilde{a}(z)=K\exp\big(-\frac{1}{c}\int^z_{z_0}a(s){\rm d}s\big)$ as $z\to+\infty$.
Note that $za(z) \leq(1-\delta_1)c$ for some $\delta_1\in (0,1)$ and $z\gg1$, then
$\widetilde{\sigma}_2(z)\sim K_1\exp\big(\frac{-1}{c}\int_{z_0}^za(s){\rm d}s\big)\geq K_2z^{\delta_1-1}$ for~$z\gg1$ and some positive constants $K_1,K_2$, thus
$\widetilde{\sigma}_2(z)\notin L^1([z_0,+\infty))$.

For the case (iii) in Remark \ref{rem2.1}, only note that
\begin{equation*}
\begin{array}{l}
a(z)\le c\Big(\frac{1}{z}+\frac{1}{z\ln z}_\cdot\cdot\cdot+\frac{1}{z\ln z\cdot\cdot\cdot\ln^kz}\Big)=c\displaystyle \frac{{\rm d}}{{\rm d}z}\Big(\ln z+\ln^2z+\cdot\cdot\cdot+\ln^{k}z+\ln^{k+1}z\Big)\\
=c\displaystyle \frac{{\rm d}}{{\rm d}z}\Big(\ln |z\ln z\cdot\cdot\cdot\ln^{k-1}z(\ln^kz)|\Big),~~z\gg1,
\end{array}
\end{equation*}
and $\widetilde{\sigma}_2(z)\sim K\exp\big(-\frac{1}{c}\int^z_{z_0}a(s){\rm d}s\big)$ as $z\to+\infty$, and
$$\exp\Big(-\frac{1}{c}\int^z_{z_0}a(s){\rm d}s\Big)\geq \frac{M_1}{|z\ln z\cdot\cdot\cdot (\ln^kz)|}\notin L^1([z_0,+\infty)), \;z\gg 1,~M_1>0.
  $$
\end{proof}

\begin{thm}\label{damvt2}
Assume that ({\bf $H_1$}) is satisfied with some  $c>0$, and
 $a(z)$ satisfies
\begin{equation}\label{z09}
\displaystyle\int^{+\infty}_{z_0}\widetilde{\sigma}_2(z){\rm d}z<+\infty, ~~
\widetilde{\sigma}_2(z): = \exp\left(\int_{z_0}^z \frac{-c+\sqrt{c^2 - 4a(s)}}{2}\, {\rm d}s\right),~ z\ge z_0\gg 1.
\end{equation}

(i) (\ref{bz}) admits a family of local positive solutions $\psi(z)$ decaying like
\begin{equation}\label{cv}\psi(z)\sim K \widetilde{\sigma}_2(z)\in L^1([z_0,+\infty)),~as~z\rightarrow+\infty,
\end{equation}
for some constant $K>0$ and there exist no other types of non-exponential local positive solutions belonging in $L^1([z_0,+\infty))$.

 (ii)  (\ref{bz}) may also admit another type of  local positive solutions not in $L^1([z_0,+\infty))$, such solution is denoted by $\psi_m(z)$, which  must satisfy
\begin{equation}\label{obac}\psi_{m}(z)\sim \frac{c\widetilde{\sigma}_2(z)}{\int_{z}^{+\infty}\widetilde{\sigma}_2(s){\rm d}s}=
-c\displaystyle \frac{{\rm d}}{{\rm d}z}\left[ \ln \left( \int_z^{+\infty} \widetilde{\sigma}_2(s) {\rm d}s \right) \right]\notin L^1([z_0,+\infty)),~as~z\to+\infty.\end{equation}
\end{thm}
\begin{proof}
 Note that
$\sigma_2(z)=-\frac{a(z)}{c}+O(a^2(z))$ for $z\geq z_0\gg 1$,
 then the assumption  \eqref{z09} is equivalent to  the following assumptions :
\begin{equation}\label{z10}
\int^{+\infty}_{z_0}a(z){\rm d}z=+\infty\;\;{\rm and}\; \displaystyle\int^{+\infty}_{z_0}\widetilde{\sigma}_2(z){\rm d}z<+\infty.
\end{equation}
By virtue of \eqref{z10}, the statement  (i) follows directly from  the standard  arguments based on   more general  ODE asymptotic theories as proved in Section 2, here we omit the detailed proof.

 (ii) Under the assumption (\ref{z10}),
let $\psi_m(z)$ be a local positive solution of (\ref{bz}) defined on $[z_0,+\infty)$ and
$\psi_m(z)\notin L^1([z_0,+\infty))$,
 by checking the proof of Theorem \ref{deca}, it is easy to see that the estimates (\ref{eq:3.2})-(\ref{eq:3.7}) are still valid for $\psi_m(z)$ on \([z_0, +\infty)\), i.e.
$\psi_m(z)$ also satisfies the integral equation
\begin{align}\label{le3.21}
\psi_m(z) = v(z) \widetilde{\sigma}_2(z) \exp\left(\int_{z_0}^z \frac{\psi_m(s)}{c + b(s)} \, {\rm d}s\right), \quad z \geq z_0,
\end{align}
with $v(z) \to k_0 > 0$ and $b(z) \to 0$ as $z \to +\infty$.

Let
\begin{align}\label{eq:de}
\psi_\ast(z) = \frac{\psi_m(z)}{c + b(z)}~~{\rm and}~~ \Psi_\ast(z) = \int_{z_0}^z \psi_\ast(s) \, {\rm d}s, \quad z \geq z_0,
\end{align}
obviously the assumption of (\ref{z10}) implies $\Psi_\ast(+\infty) =\int_{z_0}^{+\infty}\psi_\ast(s){\rm d}s=+\infty$, and the integral equation (\ref{le3.21}) can be rewritten as
\begin{equation}\label{le3.23}
\exp(-\Psi_\ast(s))\Psi_\ast'(s)=\frac{v(s)}{c+b(s)}\widetilde{\sigma}_2(s), \quad s \geq z_0.
\end{equation}
Integrating the equation (\ref{le3.23}) from $z$ to $+\infty$, we have
\begin{equation*}
\exp(-\Psi_\ast(z)) - \exp(-\Psi_\ast(+\infty)) = \int_z^{+\infty} \frac{v(s) \widetilde{\sigma}_2(s)}{c + b(s)} {\rm d}s, \quad z \geq z_0,
\end{equation*}
thus
\begin{equation}\label{psix}
\exp(\Psi_\ast(z)) = \left[\displaystyle\int_z^{+\infty} \frac{v(s) \widetilde{\sigma}_2(s)}{c + b(s)} {\rm d}s\right]^{-1}, \quad z \geq z_0,
\end{equation}
substituting (\ref{psix}) and (\ref{eq:de}) into (\ref{le3.21}), we have
\begin{equation}\label{le3.24}
\psi_m(z) = \frac{v(z) \widetilde{\sigma}_2(z)}{\displaystyle\int_z^{+\infty} \frac{v(s) \widetilde{\sigma}_2(s)}{c + b(s)} {\rm d}s}.
\end{equation}
Using the fact $v(z) \to k_0 > 0,  b(z) \to 0 \, \text{as } z \to +\infty$ and $\widetilde{\sigma}_2(z) \in L^1([z_0, +\infty))$, by (\ref{le3.24}) and some simple computations, it is easy to prove that
\begin{equation*}
\psi_m(z) \sim \frac{c \widetilde{\sigma}_2(z)}{\int_z^{+\infty} \widetilde{\sigma}_2(s) {\rm d}s}, \quad \text{as } z \to +\infty,
\end{equation*}
This completes the proof.
\end{proof}

In the following of this section, it remains to prove the existence of a local positive solution of \eqref{bz} which are not belonging in $L^1([z_0,+\infty))$ and by Theorem \ref{damvt2} such solution must satisfy the decaying estimate \eqref{obac}.

\begin{thm}\label{decay1}
 Assume that ({\bf $H_1$}) is satisfied for some $c>0$, and  assume that $\widetilde{\sigma}_2(z)\in L^1([z_0,+\infty))$, then
(\ref{bz})  admits a  local positive solution not belonging in  $L^1([z_0,+\infty))$, which is denoted by $\psi_m(z)$ and   must satisfy the  decaying estimate
\eqref{obac}.
\end{thm}
\begin{proof}
 Under the assumption $\widetilde{\sigma}_2(z)\in L^1([z_0,+\infty))$,
denote
\begin{equation}\label{le4.1}
\psi^*(z) = \frac{c \widetilde{\sigma}_2(z)}{\int_{z}^{+\infty} \widetilde{\sigma}_2(s) \, {\rm d}s} = \frac{c}{\int_{z}^{+\infty} {\rm e}^{\int^s_{z} \sigma_2(\tau)\,{\rm d}\tau}  \, {\rm d}s}.
\end{equation}
Note that $\psi^*(z) \to 0$ as $z \to +\infty$, and $\psi^*(z) \notin L^1([z_0, +\infty))$. By simple computation, it can be verified that
\begin{equation}\label{le4.2}
(\psi^*)'(z) = \left(\sigma_2(z) + \frac{1}{c} \psi^*(z)\right) \psi^*(z) \quad \text{for } z \ge z_0 \gg 1,
\end{equation}
and
\begin{equation}\label{le4.3}
0 < \psi^*(z) < -c \sigma_2(z) \quad \text{and} \quad (\psi^*)'(z) <0 \quad \text{for } z\ge z_0 \gg 1.
\end{equation}
By (\ref{le4.2}) and some computation,
\begin{equation*}
\begin{aligned}
(\psi^*)''(z)&= \left(\sigma_2(z) + \frac{2}{c} \psi^*(z)\right) (\psi^*)'(z) + \sigma'_2(z) \psi^*(z)\\
&= \left(\sigma_2^2(z) + \sigma'_2(z)\right) \psi^*(z) + \frac{3\sigma_2(z)}{c} (\psi^*(z))^2 + \frac{2}{c^2} (\psi^*(z))^3, \quad z \ge z_0.
\end{aligned}
\end{equation*}
Thus
\begin{align}\label{leequa}
& (\psi^*)''(z) + c (\psi^*)'(z) + \psi^*(z) \big(a(z) - \psi^*(z)\big) \nonumber \\
&= \left(\sigma_2^2(z) + c \sigma_2(z) + a(z) + \sigma'_2(z)\right) \psi^*(z) + \big(\psi^*(z)\big)^2 \left(\frac{3\sigma_2(z)}{c} + \frac{2}{c^2} \psi^*(z)\right) \nonumber\\
&=\frac{1}{c^2}(\psi^*(z))^2 \left[3 c\sigma_2(z) + 2 \psi^*(z)\right]+\sigma'_2(z) \psi^*(z).
\end{align}
Note that
\[
\sigma_2(z) = -\frac{a(z)}{c} - \frac{2a^2(z)}{c^3} + o(a^2(z)) < -\frac{a(z)}{c}, \quad z \ge z_0 \gg 1,
\]
and
\begin{equation}\label{le4.5}
\sigma_2'(z) = -\frac{2a'(z)}{\sqrt{c^2 - 4a(z)}} > 0, \quad \text{for } z \ge z_0 \gg 1.
\end{equation}
Let
\begin{equation}\label{le4.6}
\psi_-(z) = (1 - \delta) \psi^*(z), \quad z \ge z_0,
\end{equation}
with a fixed positive constant $\delta \in (0, 1)$.
By (\ref{le4.3})-(\ref{le4.5}) and similar computation, it can be verified that
\begin{equation}\label{eq:inequality}
\begin{aligned}
&\psi_-''(z) + c \psi_-'(z) + \psi_-(z) (a(z) - \psi_-(z))\\
&= (1 - \delta) \sigma_2'(z) \psi^*(z)
+ (1 - \delta) (\psi^*(z))^2 \left[\delta + \frac{3 \sigma_2(z)}{c^2} + \frac{2}{c^2} \psi^*(z)\right] \\
&\ge \frac{\delta}{2} (1 - \delta) (\psi^*(z))^2 > 0, \quad \text{for } z \gg 1,
\end{aligned}
\end{equation}
due to the fact  $\sigma_2(z)\to 0,\;\psi^*(z)\to 0$ as $z\to +\infty$.

For fixed $\delta>0$, we can choose a local positive function $a_+(z)$ defined on $[z_0,+\infty)$ for $z_0\gg1$, such that
\begin{equation}\label{prona}
\begin{aligned}
& a_+(z) \ge (1+\delta)a(z)~~{\rm and} \quad a_+'(z) \le 0,~ z \in [z_0,+\infty)\\
& \lim_{z \to +\infty} \frac{a_+'(z)}{a_+(z)} = 0~~{\rm and} \quad a_+(+\infty) = 0.
\end{aligned}
\end{equation}
Denote
\[
\sigma_+(z) = \frac{-c+\sqrt{c^2-4a_+(z)}}{2}, \quad \widetilde{\sigma}_+(z) = {\rm e}^{\int_{z_0}^{z} \sigma_+(s) \, {\rm d}s}, \quad z \gg 1.
\]
Using the fact that $\sigma_+(z) < \sigma_2(z) < 0$, then $\widetilde{\sigma}_+(z) \in L^1([z_0,+\infty))$. Define
\begin{equation}\label{eq:noq}
\psi_+(z) = \frac{c \widetilde{\sigma}_+(z)}{\int_{z}^{+\infty} \widetilde{\sigma}_+(s) \, {\rm d}s} = \frac{c}{\int_{z}^{+\infty} {\rm e}^{\int_{z}^{s} \sigma_+(\tau) \, {\rm d}\tau} \, {\rm d}s}.
\end{equation}
By (\ref{le4.1}), (\ref{prona})-(\ref{eq:noq}), it is easy to prove that $\psi_+(z) \to 0$ as $z \to +\infty$ and $\psi_+(z) \notin L^1([z_0,+\infty))$, and
\begin{equation}\label{Le:equa1}
\psi_+(z) \ge \psi^*(z) \ge \psi_-(z) \quad \text{for } z \gg 1.
\end{equation}

Also note that
\begin{equation}
\psi_+'(z) = \sigma_+(z) \psi_+(z) + \frac{1}{c} \psi_+^2(z), \quad z \ge z_0,
\end{equation}
and
\begin{equation*}
\begin{aligned}
&\psi_+'(z) < 0, \quad 0 < \psi_+(z) < -c \sigma_+(z), \quad z \ge z_0,\\
&\sigma_+^2(z) + c \sigma_+(z) + a_+(z) \equiv 0, \quad z \ge z_0.
\end{aligned}
\end{equation*}
By similar computation as in (\ref{eq:inequality}), we have
\begin{align}
&\psi''_+(z) + c \psi'_+(z) + \psi_+(z) (a(z) - \psi_+(z)) \notag \\
&= \left(\sigma^2_+(z) + c \sigma_+(z) + a(z)+\sigma'_+(z)\right) \psi_+(z) + \psi^2_+(z) \left[\frac{3\sigma_+(z)}{c} + \frac{1}{c^2} \psi^2_+(z)\right] \notag \\
&= \left[a(z) - a_+(z) - \frac{2a'(z)}{\sqrt{c^2 - 4a_+(z)}}\right] \psi_+(z) + \frac{\psi^2_+(z)}{c^2} \big(3c \sigma_+(z) + \psi_+(z)\big) \notag \\
&\le a_+(z) \left[\frac{-\delta}{1+\delta} + \frac{4}{c}\frac{|a'_+(z)|}{a_+(z)} \right] \psi_+(z) + \frac{2}{c} \psi^2_+(z) \sigma_+(z) \notag \\
&\le \frac{-\delta}{2(1+\delta)} a_+(z) \psi_+(z) < 0, \quad  z \gg 1, \label{eq:final_inequality}
\end{align}
using the fact
$
\frac{|a'_+(z)|}{a_+(z)} \to 0 ~ \text{as } z \to +\infty,
$
and $\sigma_+(z) < 0$ for $z \gg 1$.

Thus $\psi_+(z)$ and $\psi_-(z)$ are super/subsolutions of (\ref{bz}), which together with (\ref{Le:equa1}) proves the existence of a local positive solution
$\psi_m(z)\notin L^1([z_0,+\infty))$ satisfying
 $\psi_-(z) \le \psi_m(z) \le \psi_+(z)~{\rm for}~z \ge z_0$.
\end{proof}

\begin{proposition}\label{pro3.2}
The assumption $\widetilde{\sigma}_2(z)\in L^1([z_0,+\infty))$ includes  three typical cases of $a(z)$ as stated in Remark \ref{rem2.3}, and the decaying estimates stated in Remark \ref{rem2.3} hold true.
\end{proposition}
\begin{proof} For the  case (i): under the assumption of (\ref{crit+}),  for $r>c$ and the fixed integer $k\ge 1$, let
\begin{align}
a_0(z)&=c\left(\frac{1}{z}+\frac{1}{z\ln z}_\cdot\cdot\cdot+\frac{r/c}{z\ln z\cdot\cdot\cdot\ln^kz}\right)\\
&=c\displaystyle\frac{{\rm d}}{{\rm d}z}\left(\ln \left|z\ln z\cdot\cdot\cdot\ln^{k-1}z(\ln^kz)^\frac{r}{c}\right|\right),\;\;z\gg 1, \label{last}
\end{align}
and note that
\begin{equation}\label{sigm+}
\exp\Big(\frac{-1}{c}\int_{z_0}^za_0(s){\rm d}s\Big)=
\frac{M}{|z\ln z\cdot\cdot\cdot(\ln^kz)^\frac{r}{c}|},
\end{equation}
which with the estimates $a(z)\geq a_0(z)$ and $\sigma_2(z)\leq -a(z)/c$ for $z\gg 1$ assures that
$$
\widetilde{\sigma}_2(z)\leq  \exp\Big(\frac{-1}{c}\int_{z_0}^za_0(z){\rm d}s\Big)=
O\big(\frac{1}{z\ln z\cdot\cdot\cdot(\ln^kz)^\frac{r_1}{c}}\big),\;z\gg 1.
$$
Denote
$$
\Psi_{a}(z)=\frac{c \widetilde{\sigma}_2(z)}{\int_{z}^{+\infty}\widetilde{\sigma}_2(s){\rm d}s}
=\frac{c}{\int_{z}^{+\infty} {\rm e}^{\int^s_{z} \sigma_2(\tau)\,{\rm d}\tau}  \, {\rm d}s}.$$
Let $\sigma_0(z):=\frac{-c+\sqrt{c^2-4a_0(z)}}{2}$ and note that
$\sigma_2(z)\leq \sigma_0(z)\leq -\frac{1}{c}a_0(z)$, further by (\ref{last}) and \eqref{sigm+}, it is easy to verify that
$$
\Psi_a(z)\ge \Psi_{a_0}(z)\sim \frac{c}{\int_{z}^{+\infty} {\rm e}^{\frac{-1}{c}\int^s_{z} a_0(\tau)\,{\rm d}\tau}  \, {\rm d}s}
=(r-c)\big(z\ln z\cdot\cdot\cdot(\ln^kz)\big)^{-1},\;\;z\gg 1.
$$

 For the case (ii): $\varliminf\limits_{z\to+\infty}za(z)=\gamma\in (c,+\infty)$, let $a_\gamma(z)=\frac{\gamma}{z}$, it is easy to check that
$\Psi_{a_\gamma}(z)= \frac{\gamma -c}{z}$,
then the estimates hold true by comparison argument.

For the case (iii): $\lim\limits_{z\to+\infty}za(z)=+\infty$,
 similar to case (ii), we have   $\Psi_{a}(z)\geq \frac{p-c}{z},~\forall p>1$, which means
$z\Psi_a(z)\to +\infty$ as $z\to +\infty$. Note that if  $\lim\limits_{z\to +\infty}\frac{a'(z)}{a^2(z)}=0$, then
\begin{equation*}\label{q}
\begin{aligned}
\lim_{z\to+\infty}\frac{\Psi_{a}(z)}{a(z)}&=
\lim_{z\to+\infty}\frac{c\big{(}\widetilde{\sigma}_2(z)/a(z)\big{)}'}{-\widetilde{\sigma}_2(z)}=
\lim_{z\to+\infty}\Big{(}\frac{-c\widetilde{\sigma}_2'(z)}{\widetilde{\sigma}_2(z)a(z)}+\frac{ca'(z)}{a^2(z)}\Big{)}\\
&=\lim_{z\to+\infty}\frac{-c\sigma_2(z)}{a(z)}+\lim_{z\to+\infty}\frac{ca'(z)}{a^2(z)}=1.
\end{aligned}
\end{equation*}
\end{proof}

\begin{remark}\label{rem3.1}
From the proofs of Propositions  \ref{prop3.1} and \ref{pro3.2}, we can see that for some typical  critical subcases when  $\lim\limits_{z\to+\infty}za(z)=c$, the second and higher order asymptotic expansion of $a(z)$ at $z=+\infty$  heavily affect the decaying rate of $\widetilde{a}(z)=\exp\big{(}-\frac{1}{c}\int_{z_0}^{z}a(s){\rm d}s\big{)}$ and $\widetilde{\sigma}_2(z)$, and $\int^{+\infty}_{z_0}\widetilde{\sigma}_2(z){\rm d}z$ can be convergent or divergent, which  further determine the existence/nonexistence of local positive solutions of (\ref{bz}) decaying non-exponentially at $z=+\infty$. It is easy to see that there exist more general critical cases of $a(z)$  not included in the typical  cases of Remark  \ref{rem2.1} or Remark \ref{rem2.2};  such as $\lim\limits_{z\to+\infty}za(z)=c$ but neither  (\ref{crit-}) nor (\ref{crit+}) is  satisfied, or  $\limsup\limits_{z\to+\infty}za(z)> c>\liminf\limits_{z\to+\infty}za(z)$. For such non-typical cases or more general critical cases of $a(z)$, Theorems \ref{deca}-\ref{decay1} indicate that the existence/nonexistence and the precise decay of local positive non-exponential decay solutions of (\ref{bz})  are crucially determined  (classified) by the test $\int^{+\infty}_{z_0}\widetilde{\sigma}_2(z){\rm d}z$ is convergence/divergence, which further requires the precise expression of $a(z)$ and precise asymptotic behavior of $a(z)$ near $z=+\infty$. The local existence/nonexistence and the precise decay obtained in Section 3 are also useful or crucial for our further investigation of the existence/nonexistence and the multiplicity/uniquesness of forced waves.
\end{remark}

\section{Existence and uniqueness/multiplicity  of  forced traveling waves}

In this section, we focus on the existence/nonexistence, the uniqueness/multiplicity and spatial decay of all the forced waves $\phi_c(x-ct)$ of (\ref{Forceddd11}), where $\phi_c(z)$ are the global
positive solutions to the following differential equation:
\begin{equation}\label{nonle4}
\phi''_{c}(z)+c\phi'_c(z)+\phi_c(z)\big{(}a(z)-\phi_c(z)\big{)}=0,~z\in \mathbb{R},
\end{equation}
and satisfy the asymptotic boundary condition
\begin{equation}\label{nokm}
\phi_c(-\infty)=\alpha>0,~~\phi_c(+\infty)=0.
\end{equation}
By virtue of the existence/nonexistence and spatial decay of all the local positive solutions $\psi(z)$ of (\ref{bz}) obtained in Section 3,
we shall prove the existence
of several types  of wave solutions by constructing appropriate pairs of super- and sub-solutions.
Furthermore, the uniqueness of the forced wave with exponential decay or with the slowest non-exponential decay can be proved by applying the
variational method, sliding technique and comparison argument.

\begin{lemma}\label{extence31}
For any given $c>0$,  under the assumption
$\widetilde{\sigma}_2(z)\in L^1([z_0,+\infty))$,
(\ref{nonle4}) and (\ref{nokm}) admit at most one forced wave solution $\phi_c(z)$  not belonging in $L^1([z_0,+\infty))$.
\end{lemma}
\begin{proof}
Assume that (\ref{nonle4}) and ((\ref{nokm})) admit two distinct  forced wave solutions not belonging in $L^1([z_0,+\infty))$, which are  denoted by  $\varphi(z)$ and $\psi(z)$.
Suppose that $\varphi(z)<\psi(z),~\forall z\in\mathbb{R}$, owing to the ordering of forced waves. Define
$$k_0=\inf K:=\{k>0:~k\varphi(z)\geq\psi(z),~\forall z\in\mathbb{R}\}. $$

 By Lemma \ref{damvt2}, we have
$$
\varphi(z)\sim \psi(z)\sim\frac{c\widetilde{\sigma}_2(z)}{\int_{z}^{+\infty}\widetilde{\sigma}_2(s){\rm d}s},~{\rm as}~z\to+\infty,
$$
and by virtue of $\varphi(-\infty)=\psi(-\infty)=\alpha>0$, we have
$$
\lim_{z\to\pm\infty}\frac{\psi(z)}{\varphi(z)}=1.
$$
Then there exists some constant $M>0$ such that $\frac{\psi(z)}{\phi(z)}\leq M$ for all $z\in\mathbb{R}$,
which implies that the set $K$ is not empty and $k_0>1$.
If $k_0=1$,  then $\varphi(z)\geq\psi(z),~\forall z\in\mathbb{R}$,  which is impossible.

Let $w(z)=k_0\varphi(z)-\psi(z)\geq0,~~z\in\mathbb{R}$,
then $\inf\limits_{z\in\mathbb{R}}w(z)=0$.
Using the fact that $\frac{\psi(z)}{\varphi(z)}\to1$ as $z\to+\infty$ and $k_0>1$,  then there
exists large $z_0>0$ such that $w(z)>0$ for $z>z_0$,  and note that
$w(-\infty)=(k_0-1)\alpha>0$. Thus, $w(z)$ admits a local minimum $w(\overline{z})=0$
for some $\overline{z}\in\mathbb{R}$.  Further, by (\ref{nonle4}), we obtain
\begin{equation}\label{m}
\begin{aligned}
w''(z)+cw'(z)&=-k_0\varphi(z)\big{(}a(z)-\varphi(z)\big{)}+\psi(z)\big{(}a(z)-\psi(z)\big{)}\\
&\leq \big{(}k_0\varphi(z)+\psi(z)-a(z)\big{)}\big{(}k_0\varphi(z)-\psi(z)\big{)}=:m(z)w(z),~z\in\mathbb{R},
\end{aligned}
\end{equation}
where $m(z)$ is bounded. Due to the fact $w(z)\geq 0$ for $z\in \mathbb{R}$,
by applying the strong maximum principle to (\ref{m}), we have $w(z)\equiv0$ on $z\in\mathbb{R}$,
which contradicts to $w(-\infty)>0$. This completes the proof.
\end{proof}

\begin{prop}\label{ex}
Assume that $c>0$ and  ({\bf $H_1$}) is satisfied,
if (\ref{nonle4}) and \eqref{nokm}  admit a forced wave solution $\phi_c(z)$ decaying exponentially  as $z\to+\infty$, then
$c\in(0,~2\sqrt{\alpha})$. \end{prop}

\begin{proof}
By contradiction, suppose that $c\geq2\sqrt{\alpha}$, then by Lemma \ref{decavr} $\phi_c(z)$ satisfies
$$
\phi_c(z)=o\big{(}{\rm e}^{(-c+\epsilon)z}\big{)},~{\rm as}~z\to+\infty,~~0<\epsilon\ll1.
$$
Set $W(z)={\rm e}^{\frac{c}{2}z}\phi_c(z)$, then $W(z)$ satisfies
\begin{equation}\label{eigretrans12}
\begin{cases}
&W''(z)+\Big{[}a(z)-\phi_c(z)-\frac{c^2}{4}\Big{]}W(z)=0,~~\forall z\in\mathbb{R},\\
&W(\pm\infty)=0.
\end{cases}
\end{equation}
Note that $a(z)-\phi_c(z)-\frac{c^2}{4}<\alpha-\frac{c^2}{4}\leq0$ for any $z\in\mathbb{R}$ if $c\geq2\sqrt{\alpha}$,
then $W''(z)>0$ on $z\in\mathbb{R}$, which contradicts to $W(z)>0$ on $z\in\mathbb{R}$ and $W(\pm\infty)=0$. \end{proof}

\begin{lemma}\label{extence1}
For any given $c\in(0,~2\sqrt{\alpha})$, (\ref{nonle4}) and \eqref{nokm} admit
at most one forced wave solution decaying exponentially, i.e., $\phi_c(z)=O\big{(}{\rm e}^{-\sigma z}\big{)}$
as~$z\rightarrow+\infty$ for some $\sigma>0$. \end{lemma}

\begin{proof}
Lemma \ref{decavr} guarantees that
\begin{equation}\label{plg}
\phi_c(z)=o\big{(}{\rm e}^{-(c-\epsilon) z}\big{)},~{\rm as}~z\to+\infty,\;\;0<\epsilon\ll 1.
\end{equation}
Set $W(z)={\rm e}^{\frac{c}{2}z}\phi_c(z)$,  it is easy to see that $W(z)$ satisfies
\begin{equation}\label{Wrok}
\begin{cases}
&W''(z)+\Big{[}a(z)-{\rm e}^{\frac{-c}{2}z}W(z)-\frac{c^2}{4}\Big{]}W(z)=0,~~ z\in\mathbb{R},\\
&W(\pm\infty)=0,~~W'(\pm\infty)=0.
\end{cases}
\end{equation}
Suppose that $\phi^i_c(z)(i=1,2)$ are two forced wave solutions of (\ref{nonle4}) and \eqref{nokm} satisfying (\ref{plg}), set $W_i(z)={\rm e}^{\frac{c}{2}z}\phi^i_c(z)(i=1,2)$ for $z\in\mathbb{R}$,
using the ordering of wave solutions,
we may assume $W_1(z)<W_2(z)$. By (\ref{Wrok}), it is easy to verify that
\begin{equation}\label{Wronskian}
\begin{aligned}
\big{(}W'_1(z)W_2(z)-W'_2(z)W_1(z)\big{)}'&=W_1''(z)W_2(z)-W_1(z)W''_2(z)\\
&={\rm e}^{-\frac{c}{2}z}W_1(z)W_2(z)(W_1(z)-W_2(z)).
\end{aligned}
\end{equation}
Integrating (\ref{Wronskian}) from $-\infty$ to $+\infty$, and using the fact that $W_i(\pm\infty)=0=W'_i(\pm\infty)(i=1,2)$,
we have
\begin{equation*}
\begin{aligned}
0
=\int_{-\infty}^{+\infty}{\rm e}^{-\frac{c}{2}z}W_1(z)W_2(z)(W_1(z)-W_2(z)){\rm d}z<0,
\end{aligned}
\end{equation*}
which is a contradiction. This completes the proof.
\end{proof}

In the following, we investigate the existence of forced waves to (\ref{nonle4}) and \eqref{nokm} for any  $a(z)$ satisfying ({\bf $H_1$})
and one can see that
the existence of multi-type wave solutions $\phi_c(z)$ to (\ref{nonle4})-\eqref{nokm} and the construction of super-/sub-solution highly depends on the decay rate of $a(z)$ as $z\rightarrow+\infty$.

{\bf Proof of Theorem \ref{thmfirst1}}

To prove the statement (I), we shall
apply sub-/super- solution method  to prove the existence of  a forced
wave $\phi_c(z)$ of  (\ref{nonle4})-\eqref{nokm} decaying exponentially  at $z=+\infty$.

We adopt some basic ideas in \cite{[BDD]} to construct a weak nonnegative  sub-solution of (\ref{nonle4}).
For each fixed $c\in (0,2\sqrt{\alpha})$, the limit $\lim\limits_{z\to -\infty}a(z)=\alpha$ guarantees that there exist positive constants $a_0>0$ and large enough $L\gg 1$ such that
$$
a(z)-\frac{c^2}{4}\ge 2a_0>0,~~{\rm for}\;\;z\leq -\frac{L}{2};\;\;{\rm and}\;\;\frac{\pi^2}{L^2}\le a_0,
$$
 thus
\begin{equation}\label{4.1a}
a(z)\ge \frac{c^2}{4}+ a_0+\frac{\pi^2}{L^2},~~\;\;z\leq -\frac{L}{2}.
\end{equation}
Define
\begin{equation}\label{4.1b}
\underline{u}(z)=
\left\{\begin{aligned}
&\delta {\rm e}^{-\frac{c}{2}z}W_L(z),
~~~~-\frac{3L}{2}<z<-\frac{L}{2},\\
&~0,~~~~~~~~~~~~~~~~~~~|z+L|\geq \frac{L}{2},
\end{aligned}
\right.
\end{equation}
with $W_L(z)=\cos\big(\frac{z\pi}{L}+\pi\big)$. Note that $W''_L(z)+\frac{\pi^2}{L^2}W_L(z)\equiv 0$.

 By virtue of  \eqref{4.1a} and choosing $\delta>0$ small enough such that $\delta\le a_0{\rm e}^{-\frac{3cL}{4}}$, then for  $-\frac{3L}{2}<z<-\frac{L}{2}$, it is easy to check that
\begin{equation*}
\begin{aligned}
&\underline{u}''(z)+c\underline{u}'(z)+a(z)\underline{u}(z)-\underline{u}^2(z)\\
=&\delta{\rm e}^{-\frac{c}{2}z}\left( W''_L(z)+\Big( a(z)-\frac{c^2}{4}\Big)W_L(z)-\delta {\rm e}^{-\frac{c}{2}z}W_L(z)\right)\\
\geq& \delta {\rm e}^{-\frac{c}{2}z} \left(W''_L(z)+\frac{\pi^2}{L^2}W_L(z)+\left(a_0-\delta{\rm e}^{-\frac{c}{2}z}W_L(z)\right)W_L(z)\right)\ge 0.
\end{aligned}
\end{equation*}
Obviously zero is sub-solution of (\ref{nonle4}), thus $\underline{u}(z)$ defined in \eqref{4.1b} is a nonnegative  sub-solution of (\ref{nonle4}).

Note that for any given $\epsilon>0$,  there exists $\overline{z}_\epsilon>0$ and large $k>0$
such that
$a(z)\leq\epsilon(c-\epsilon)$, $k{\rm e}^{-(c-\epsilon)\overline{z}_\epsilon}=\alpha$ and $k{\rm e}^{-(c-\epsilon)z}<\alpha$ for $z>\overline{z}_\epsilon$,
where $a(z)\leq\alpha,~\forall z\in\mathbb{R}$.

Let $\bar{u}_1(z)=k{\rm e}^{-(c-\epsilon)z}$, then
\begin{equation*}
\begin{aligned}
&\bar{u}_1''(z)+c\bar{u}_1'(z)+a(z)\bar{u}_1(z)-\bar{u}_1^2(z)\\
=&\bar{u}_1(z)\big{(}-\epsilon(c-\epsilon)+a(z)-\bar{u}_1\big{)}\leq 0,~~z>\overline{z}_\epsilon,
\end{aligned}
\end{equation*}
obviously $\overline{u}(z):=\min\{\alpha,~k{\rm e}^{-(c-\epsilon)z}\}$
is also a super-solution of (\ref{nonle4}) for large enough $k>0$, and $\underline{u}(z)<\overline{u}(z)$~for $z\in\mathbb{R}$.

Applying sub/super solution method, the equation \eqref{nonle4} admits a global positive solution $\phi_c(z)$ satisfying $\underline{u}(z)\leq\phi_c(z)\leq \overline{u}(z)$.
By Lemma \ref{negnega1} and  the fact that $\phi_c(z)\leq \overline{u}(z)$, $\phi_c(z)$ must satisfy the asymptotic boundary condition \eqref{nokm}, i.e. $\phi_c(-\infty)=\alpha$ and $\phi_c(+\infty)=0$. Further by Lemma \ref{decavr} guarantees both the exponential decay of $\phi_c(z)$ as $z\to+\infty$ and the asymptotic behavior
\begin{equation*}
\phi_c(z)\sim A_1\widetilde{\sigma}_1(z),\quad z\to+\infty,
\end{equation*}
for some $A_1>0$. The uniqueness of such forced wave follows from  Lemma \ref{extence1}, which completes the proof of Theorem \ref{thmfirst1} (i).

The statement (II) for $c\geq 2\sqrt{\alpha}$ follows directly from  Proposition \ref{ex}.

Under the assumption of \eqref{214}, the statement (III) follows from Theorem \ref{deca}. This completes the proof of Theorem \ref{thmfirst1}.

{\bf Proof of Theorem \ref{thmfirspk}}
Under the assumption (\ref{plh}),
we shall show that
the problem (\ref{nonle4}) has infinitely many and ordered forced waves $\phi_c(z)$ decaying non-exponentially at $z=+\infty$.

Let $b(z)=\int_z^{+\infty}\widetilde{\sigma}_2(s){\rm d}s$
 and note that
\begin{equation}\label{btv}b'(z)=-\widetilde{\sigma}_2(z)~{\rm and}~b''(z)=-\sigma_2(z)\widetilde{\sigma}_2(z).\end{equation}
For any fixed $A>0$, choosing positive constants $M$ and $z_M$ large enough such that
$$M>\frac{2A}{c},~~Mb(z_M)=1~~{\rm and}~~\sigma_2(z)>-\frac{c}{6},~z>z_M,$$
we define
\begin{equation}\label{subsolution1}
\underline{u}(z)=
\left\{\begin{aligned}
&A\widetilde{\sigma}_2(z)(1-Mb(z)),
~~~~~~z>z_M ,\\
&0,~~~~~~~~~~~~~~~~~~~~~~~~~~~~~z\leq z_M,
\end{aligned}
\right.
\end{equation}

By simple computation
\begin{equation*}
\begin{aligned}
\underline{u}'(z) &= A \widetilde{\sigma}_2(z) \left[ \sigma_2(z)(1-M b(z)) + M \widetilde{\sigma}_2(z) \right], \quad z > z_M,\\
\underline{u}''(z) &= A \widetilde{\sigma}_2(z) \left[ (\sigma_2^2(z) + \sigma_2'(z))(1 - M b(z)) + 3M \sigma_2(z) \widetilde{\sigma}_2(z) \right], \quad z > z_M.
\end{aligned}
\end{equation*}
which with (\ref{btv}) yields that
\begin{equation*}
\begin{aligned}
&\underline{u}''(z) + \underline{u}'(z) + \underline{u}(z)(a(z) - \underline{u}(z)) \\
&= A \widetilde{\sigma}_2(z) \left[ \big(\sigma_2^2(z) + c \sigma_2(z) + a(z)\big)(1 - M b(z)) \right]
+ A \widetilde{\sigma}_2(z) \sigma_2'(z) (1 - M b(z)) \\
&\quad + A \left[ \widetilde{\sigma}_2(z) \right]^2 \left[ M(c + 3 \sigma_2(z)) - A (1 - M b(z))^2 \right] > 0,\quad z > z_M;
\end{aligned}
\end{equation*}
noting that
$
\sigma_2^2(z) + c \sigma_2(z) + a(z) \equiv 0$ for $z \geq z_M,
$
and
\begin{equation*}
\sigma_2'(z) = \frac{-2a'(z)}{\sqrt{c^2 - 4a(z)}} > 0,~~
M(c + 3 \sigma_2(z)) \geq \frac{c}{2} M>A \geq A (1 - M b(z))^2,\quad z > z_M.
\end{equation*}
Thus, $\underline{u}(z)$ is a non-trivial sub-solution of  (\ref{nonle4}) decaying  non-exponentially and decaying like  $A \widetilde{\sigma}_2(z)$ as $z \to +\infty$. Obviously  $\overline{u}(z)\equiv \alpha$  is a super solution of (\ref{nonle4}), which with $\underline{u}(z)$ guarantees the existence of  a global positive and bounded classical solution $\phi_c(z)$ of (\ref{nonle4}) satisfying $\underline{u}(z)\leq\phi_c(z)\leq\overline{u}(z)$.
Lemma \ref{negnega1}guarantees that $\phi_c(z)$ must be a forced wave  satisfying  (\ref{nonle4}) and (\ref{nokm}), which proves that there exists at least one
forced wave decaying non-exponentially at $z=+\infty$.

Further by  Theorem \ref{thmfirst1}(I) and Theorem A it follows that for $c\in(0,2\sqrt{\alpha})$ except the minimal forced wave decaying exponentially at $z=+\infty$ there exist
infinitely many forced waves $\phi_c(z)$ of (\ref{nonle4}) and (\ref{nokm}), which  decay non-exponentially at $z=+\infty$.
 While for $c\geq 2\sqrt{\alpha}$, using the fact that $\lambda_1\ge -\alpha+\frac{c^2}{4}\ge 0$ by virtue of  $a(z)\in (0,\alpha]$, then by Theorem A (ii) and Theorem \ref{thmfirst1} (II) it follows that (\ref{nonle4}) and (\ref{nokm}) have neither the minimal forced wave nor a wave decaying exponentially at $z=+\infty$, the above facts imply that (\ref{nonle4}) and (\ref{nokm}) admit  infinitely many forced waves and all the
waves decaying non-exponentially at $z=+\infty$.

 For any fixed $c>0$ let $\bar{u}(t,z):=u(t,z;\overline{u}_0)$ be the unique solution to the Cauchy problem of the  parabolic equation $u_t=u_{zz}+cu_z+u(a(z)-u)$ with $\overline{u}_0(z)=\alpha=\overline{u}(z)$.
By applying the standard comparison argument, it is easy to prove that $\bar{u}(t,z)$ is decreasing in $t$ for any $t>0$ and $\bar{u}(t,z)\ge \underline{u}(z)$ for any $t>0$ and $z\in \mathbb{R}$, thus $\lim\limits_{t\to +\infty}\bar{u}(t,z)=\phi^m(z)$ exists, which  is a forced wave solution of (\ref{nonle4})-(\ref{nokm}) by Lemma \ref{negnega1}, and $\phi^m(z)$ must be  the maximal forced wave $\phi_c^{\max}(z)$ if there exist other forced waves. This also proves that under the assumption of (\ref{plh}) the maximal forced wave $\phi_c^{\max}(z)$ exists and decay non-exponentially at
$z=+\infty$.

Next, we show that the maximal forced wave  $\phi^{\max}_c(z)$ is not in $L^1([z_0,+\infty))$, by contradictory argument, assume that $\phi^{\max}_c(z)\in L^1([z_0,+\infty))$, then it holds that there exists a constant $A_0>0$ such that $\phi^{\max}_c(z)\sim A_0\widetilde{\sigma}_2(z)$ as $z\to +\infty$. However the sub-solution  $\underline{u}(z)$ defined in \eqref{subsolution1}  with  $A>A_0$  implies that there exists at least one forced wave which decays more slowly than $\phi^{\max}_c(z)$ as $z\to +\infty$, which  is  a contradiction.
 Theorem \ref{damvt2} and Lemma \ref{extence31} further guarantees that  under the assumption of  (\ref{plh})   the maximal forced wave $\phi^{\max}_c(z)$ is also the unique forced wave not belonging in $L^1([z_0,+\infty))$ and must  satisfy the following  decaying estimate
\begin{equation*}\label{K}
\phi^{\max}_c(z)\sim \frac{c\widetilde{\sigma}_2(z)}{\int_{z}^{+\infty}\widetilde{\sigma}_2(\tau){\rm d}\tau},~as~z\to+\infty.
\end{equation*}
Based on the above facts  and by virtue of Lemma \ref{damvt2} it follows that  (\ref{nonle4}) and (\ref{nokm})  admit infinitely many ordered forced waves $\phi_c(z)$ decaying non-exponentially at $z=+\infty$ and $\phi_c(z) \sim K\widetilde{\sigma}_2(z)$ as $z\to +\infty$ for some constant $K>0$. This completes the proof of Theorem \ref{thmfirspk}

\section*{Appendix }

{\bf Proof of Lemma \ref{negnega1}.}\\
Let $\phi_c(z)$ be any given global positive solution of (\ref{questt1}), by applying the same argument as in the proof  \cite[Lemma A.1]{[CTW]}  it can be proved that $\phi_c(z)$ must be uniformly bounded in $\mathbb{R}$ and  $\sup\limits_{z\in \mathbb{R}}\phi_c(z)\le \sup\limits_{z\in \mathbb{R}} \{a(z)\}=\alpha$. The proof of  $\phi_c(-\infty)=\alpha>0$ can be found in
\cite[Lemma 3.1]{[BF]}. \\
We first show that $\phi'_c(z)$ has no  local positive maximum in $(z_0,+\infty)$.  By contradiction, we assume that
there exists some point $z_1>z_0$
such that   $\phi'_c(z_1)$ is a local positive maximum of $\phi'_c(z)$ for $z_1>z_0$, then $\phi''_c(z_1)=0$ and $\phi'''_c(z_1)\leq0$. Note that
\begin{equation*}
c\phi'_c(z_1)=[\phi_c(z_1)-a(z_1)]\phi_c(z_1)>0,
\end{equation*} thus,
$a(z_1)<\phi_c(z_1)$.  Thanks to $a'(z)\leq 0$ for $z>z_0$, differentiating the equation (\ref{questt1}) for $\phi_c(z)$,  we can see
$$
\phi'''_c(z_1)=\phi'_c(z_1)\Big{(}2\phi_c(z_1)-a(z_1)\Big{)}-a'(z)\phi_c(z)\big{|}_{z=z_1}>0,
$$
it is a contradiction. The nonexistence of local positive maximum of $\phi'_c(z)$ for $z\gg 1$  further implies  that $\phi'_c(z)$ is not sign changing for $z\gg 1$. Thus $\phi'_c(z)\ge 0$ for $z\gg 1$ or $\phi'_c(z)\leq 0$ for $z\gg 1$, which with the boundedness of $\phi_c(z)$ assures that
 $\lim\limits_{z\to+\infty}\phi_c(z)$ exists and $\lim\limits_{z\to+\infty}\phi'_c(z)=0$.  Now we can  prove that  $\lim\limits_{z\to+\infty}\phi_c(z)=0$, otherwise assume that $\lim\limits_{z\to+\infty}\phi_c(z)=\delta_0>0$, which with  (\ref{questt1}) and $\phi'_c(+\infty)=a(+\infty)=0$ implies that $\phi^{\prime\prime}_c(z)\ge  \delta_1>0 $ for $z\gg 1$, which contradicts  with  $\lim\limits_{z\to+\infty}\phi'_c(z)=0$. Thus $\phi_c(+\infty)=0$,  which further implies  $\phi'_c(z)\leq 0$ for $z\gg 1$.
It remains to prove that $\phi'_c(z)<0$ for $z\gg1$.  To see this, differentiating the equation (\ref{questt1}) with respect to $z$,
we have
$$\phi'''_c(z)+c\phi''_c(z)+\big{(}a(z)-2\phi_c(z)\big{)}\phi'_c(z)=-a'(z)\phi_c(z)\geq0,~~z\gg 1. $$
Applying the strong maximum principle, one has $\phi'_c(z)<0$ for  $z\gg 1$, which completes the proof.\\

{\bf  Proof of Proposition \ref{lepro}.}\\
Let $\psi(z)$ be a local positive solution of (\ref{questt1}) define on $[z_0,+\infty)$ for some large $z_0\gg 1$ and assume that $\psi(z)$ decays to zero non-exponentially at $z=+\infty$, then
 the statement of $\psi'(z)<0$ for $z\gg 1$ follows directly from the  proof of Lemma \ref{negnega1}. Denote $\theta(z)=\frac{\psi'(z)}{\psi(z)}$ and $r(z)=\frac{\psi''(z)}{\psi'(z)}$ for $z\ge z_0$;  if $r(+\infty)=0$ holds true,  then the nonlinear differential equation (\ref{questt1}) can be rewritten as
$$
\psi'(z)=\frac{\psi(z)(\psi(z)-a(z))}{c+r(z)},~~z\geq z_0,
$$
which with the assumption $r(+\infty)=0$ and the fact that $\psi'(z)<0,\;\psi(z)>0$ for $z\geq z_0\gg1$ implies that
$\psi(z)<a(z)$ for $z\geq z_0\gg1$.

It remains to prove $\theta(+\infty)=r(+\infty)=0$, the idea of the proof is similar to  {\cite[Lemma 3.3]{[BO]}}.  In the following we just give the proof of   $\theta(+\infty)=0$, for the claim    $r(+\infty)=0$ can be proved by the similar argument.

Multiplying (\ref{questt1}) by $\frac{1}{\psi(z)}$ and using the fact that
$\psi''/\psi=\theta'+\theta^2$, we have
\begin{equation}\label{cpm}\theta^2(z)+\theta'(z)+c\theta(z)=\psi(z)-a(z)\to 0,~{\rm as}~z\to+\infty.\end{equation}
We claim that $\theta(z)$ is bounded for  $z\gg 1$. If not, suppose on the contrary that $\theta(z)$ is unbounded near $z=+\infty$,
then $\theta(z)$ is either oscillating or monotone near $z=+\infty$.
In the case when $\theta(z)$ is oscillating near $z=+\infty$, then there exists a sequence $\{z_j\}_{j\in\mathbb{N}}$ such that
$|\theta(z_j)|\rightarrow+\infty$~and~$\theta'(z_j)=0$~with~$j\rightarrow+\infty$, which implies
\begin{equation*}\phi_c(z_j)-a(z_j)=\theta^2(z_j)+\theta'(z_j)+c\theta(z)\to+\infty,~{\rm as}~j\to+\infty;\end{equation*}
it is a contradiction.

In the case when $\theta(z)<0$ is monotone and unbounded for $z\gg 1$,
then $\lim\limits_{z\to+\infty}\frac{1}{\theta(z)}=0$, further by (\ref{cpm}) we have
$$
\frac{d}{dz}\big{(}\frac{1}{\theta(z)}\big{)}=\frac{-\theta'(z)}{\theta^2(z)}\rightarrow 1,\; {\rm as}\;\; z\to+\infty;
$$
which contradicts to the fact that $\lim\limits_{z\to+\infty}\frac{1}{\theta(z)}=0$. Thus, $\theta(z)$ must be  bounded for $z\gg 1$.
Based on the boundedness of $\theta(z)$ for $z\gg 1$, suppose that there exist $A$ and $B$ such that
$$
A:=\liminf_{z\to+\infty}\theta(z)\leq\limsup_{z\to+\infty}\theta(z)=:B.
$$
By contradiction, assume that $A\neq B$, then there exist two sequences $\{x_m\}_{m\in\mathbb{N}}$ and $\{y_m\}_{m\in\mathbb{N}}$ satisfying
$\lim\limits_{m\to+\infty}x_m=\lim\limits_{m\to+\infty}y_m=+\infty$
and
$$
\lim_{m\to+\infty}\theta(x_m)=A,~~\lim_{m\to+\infty}\theta(y_m)=B~~{\rm and}~~\lim_{m\to+\infty}\theta'(x_m)=\lim_{m\to+\infty}\theta'(y_m)=0;
$$
substituting them into (\ref{cpm}), we have $A=-c$ and $B=0$.  Further, choose some sequence $\{z_j\}_{j\in\mathbb{N}}$
with $\lim\limits_{j\to+\infty}z_j=+\infty$ such that
$$
\theta(z_j)=\lambda\in(-c,0)~~{\rm and}~~\theta'(z_j)<0,
$$
then by (\ref{cpm}) we have
$$
\lim_{j\to+\infty}\big{\{}\theta^2(z_j)+\theta'(z_j)+c\theta(z_j)\big{\}}\leq\lambda^2+c\lambda<0,
$$
which is impossible. Thus $\lim\limits_{z\to+\infty}\theta(z)=A=B$ and $A,B$ are real roots of $\mu^2+c\mu=0$.
Note that the local positive solution  $\psi(z)$ decays to zero non-exponentially as $z\to+\infty$, thus
$$\lim_{z\to+\infty}\theta(z)=B=\lim_{z\to+\infty}\frac{\psi'(z)}{\psi(z)}=0. $$
The proof is completed.

\section*{Acknowledgment}
The authors are very grateful to the anonymous referees for many valuable and important suggestions which helped to improve the exposition of the manuscript. We would also like to thank Professor Jian Fang for his useful suggestions and valuable discussions.
Yaping Wu and Zhibao Tang are supported by the NSF of China (No. 12371209 and No. 11871048) and Beijing NSF (No. 1232004).
Shi-Liang Wu is supported by the
NSF of China (No. 12171381), the Shaanxi Fundamental Science Research Project for Mathematics and Physics (No. 22JSY029),   and Natural Science Basic Research Program of Shaanxi (No. 2024RS-CXTD-88).

\end{document}